\documentclass[a4paper,11pt]{article}
\usepackage{stmaryrd}
\usepackage{amsfonts}
\usepackage{bbm}
\usepackage{amscd}
\usepackage{mathrsfs}
\usepackage{latexsym,amssymb,amsmath,amscd,amscd,amsthm,amsxtra,xypic}
\usepackage[dvips]{graphicx}
\usepackage[utf8]{inputenc}
\usepackage[T1]{fontenc}
\usepackage{enumerate}
\usepackage{dsfont}
\usepackage{amssymb}
\usepackage[all]{xy}
\usepackage{ifpdf}
\ifpdf
\usepackage[colorlinks=true,linkcolor=blue,final,backref=page,hyperindex,citecolor=red]{hyperref}
\else
\usepackage[colorlinks,final,backref=page,hyperindex,hypertex]{hyperref}
\fi

\usepackage{tikz-cd}
\usepackage{nicefrac,mathtools,enumitem}
\usepackage{microtype}
\usepackage{amsfonts}
\usepackage{amssymb}
\usepackage{mathrsfs}
\usepackage{amsmath}
\usepackage{amsthm}
\usepackage{enumerate}
\usepackage{indentfirst}
\usepackage{amsfonts}
\usepackage{amssymb}
\usepackage{mathrsfs}
\usepackage{amsmath}
\usepackage{amsthm}
\usepackage{enumerate}
\usepackage{cite}
\usepackage{mathrsfs}
\usepackage{geometry}
\allowdisplaybreaks
\def\Rg{\mathbb{C}[\partial]}
\def\Ep{\mathcal{E}}
\def\Vs{\mathfrak{a}}
\def\Ws{\mathfrak{b}}
\def\Mod{\mathcal{M}}
\def\rep{\varrho}
\def\str{\Omega}
\def\smu{\widehat{\nu}_1}
\def\smv{\widehat{\nu}_2}
\def\shf{\widehat{\psi}}
\def\Cch{\mathscr{C}}

\def\la{\lambda}
\def\p{\partial}
\def\C{\mathbb{C}}
\def\half{\tfrac{1}{2}}
\def\perm{\mathbb{S}}
\def\Nat{\mathbb{N}}
\newcommand{\ad}{\mathrm{ad}}

\newcommand{\dM}{\mathbf{d}}
\newcommand{\emptycomment}[1]{}

\numberwithin{equation}{section}

\newtheorem{theorem}{Theorem}[section]
\newtheorem{definition}[theorem]{Definition}
\newtheorem{corollary}[theorem]{Corollary}
\newtheorem{lemma}[theorem]{Lemma}
\newtheorem{proposition}[theorem]{Proposition}
\newtheorem{example}[theorem]{Example}
\newtheorem{remark}[theorem]{Remark}

\allowdisplaybreaks

\begin{document}
%==========================================================================

\title{\sf Deformation maps on quasi-twilled Lie conformal algebras}
\date{}
\author{ \bf 
  Taoufik Chtioui$^{1}$\footnote{E-mail: chtioui.taoufik@yahoo.fr  },
  \ \ Sami Mabrouk $^{2}$\footnote{E-mail: mabrouksami00@yahoo.fr, sami.mabrouk@fsgf.u-gafsa.tn  },\   Abdenacer Makhlouf$^{3}$
    \footnote {  E-mail:  abdenacer.makhlouf@uha.fr \text{(Corresponding author)}}
\\ \\
$^{1}${\small Mathematics and Applications Laboratory, Faculty of Sciences, Gabes University, Tunisia } \\
$^{2}${\small  University of Gafsa, Faculty of Sciences Gafsa, 2112 Gafsa, Tunisia }\\
$^{3}${\small Université de Haute Alsace IRIMAS- Département de Mathématiques, Mulhouse, 
France} 
}
\maketitle

\begin{abstract} The purpose of  this paper is to 
develop a unified approach  for various operators on 
Lie conformal algebras. Given a quasi-twilled Lie conformal algebra $(\Ep,\Vs,\Ws)$, we
introduce two dual families of operators: \emph{right deformation maps}
$D:\Vs\to\Ws$ and \emph{left deformation maps} $B:\Ws\to\Vs$. Each
family simultaneously subsumes several classical structures:
 modified $r$-matrices,  crossed homomorphisms,
 derivations, and  Lie conformal algebra homomorphisms in the
right case,  relative Rota-Baxter operators,  twisted
Rota-Baxter operators,  Reynolds operators, and deformation
maps of matched pairs in the left case. Using Voronov's derived bracket method, we construct the controlling homotopy algebras: a curved
$L_\infty$-algebra governing right deformation maps and an
$L_\infty$-algebra governing left deformation maps, with Maurer-Cartan
elements precisely characterizing each type. We further develop the
associated deformation theories via twisted $L_\infty$-algebras and
define cohomology complexes for both types of deformation maps, recovering  the cohomologies of some  classical and conformal operators already developed in the literature and defining  a new ones for modified $r$-matrices, Lie conformal algebra homomorphisms, derivations and deformation maps on matched pairs.
\end{abstract}
\tableofcontents

%==========================================================================
\section{Introduction}\label{sec:intro}
%==========================================================================

The theory of Lie conformal algebras, introduced by Kac~\cite{KAC} and
further developed by Bakalov-D'Andrea-Kac~\cite{BDK} and De~Sole-Kac~\cite{DK,DK3},
provides the algebraic language underlying the operator product expansion in
two-dimensional conformal field theory. These algebras encode the singular
part of the OPE of chiral fields, and their cohomological and
representation-theoretic properties have been the subject of active
investigation; see~\cite{BKV,DK3,DK13} for foundational results.

In parallel, the study of operator-theoretic structures on Lie
algebras:  Rota-Baxter operators, averaging operators, Reynolds
operators, modified $r$-matrices, derivations and their relatives has undergone
a renaissance, driven in large part by their connections to integrable
systems, the Yang-Baxter equation, dendriform algebras, and post-Lie
algebras. A central theme in this modern treatment, pioneered by
Tang-Bai-Guo-Sheng \cite{TBGS,TBGS2} and developed further
in \cite{Das0,Das1,JSZ,LST,TFS}, is that each class of operators should
be governed by a \emph{controlling homotopy algebra} (a (curved)
$L_\infty$-algebra or differential graded Lie algebra)  whose
Maurer-Cartan elements are precisely the operators in question, and
whose cohomology captures the infinitesimal deformation theory of those
operators. Extending this philosophy to the conformal setting is the
principal aim of this paper.

The notion of a \emph{twilled} (or bicrossed product) Lie algebra, in
which a Lie algebra splits as $\Vs\oplus\Ws$ with both summands being
subalgebras, and its generalization to \emph{quasi-twilled} Lie algebras
(where only $\Ws$ is required to be a subalgebra), provide a natural
unifying framework: by varying the quasi-twilled structure, one obtains
the semi-direct product, the $\phi$-deformed semi-direct product, and
many other standard constructions. In the conformal setting, the
structure $2$-cochain $\str$ of a quasi-twilled Lie conformal  algebra
admits a bihomogeneous decomposition
$\str = \shf + \smu + \smv$
with respect to the Nijenhuis-Richardson bigrading, and the
Maurer-Cartan equation $[\str,\str]_{NR}=0$ decomposes into a
hierarchical system of compatibility conditions. The
Nijenhuis-Richardson bracket, which underlies this framework, was
studied in the conformal context in~\cite{DK13}.

The concept of a {deformation map} for twilled Lie algebras was introduced by
Agore and Militaru~\cite{AM14a,AM14b} in their study of the
{classifying complements problem}, whose aim is to classify, up to
isomorphism, all complements of a given Lie subalgebra.
One of the remarkable features of deformation maps is that they provide a
unified framework encompassing several important classes of operators on Lie
algebras, including Lie algebra homomorphisms, derivations, crossed
homomorphisms~ \cite{BBSS}, and relative Rota-Baxter operators (or
$\mathcal{O}$-operators) of weights $0$ and $1$~\cite{Ku,Uch08}, see also \cite{HB,YL,ZBC}.
The theory of Rota-Baxter operators, originating from the work of
Baxter~\cite{Bax60} and subsequently developed by Rota~\cite{Rot69}, has
received renewed attention due to its deep connections with
Connes-Kreimer renormalization~\cite{CK98,CK00}, the classical Yang-Baxter
equation, and dendriform algebras~\cite{Guo12}. Nevertheless, several operators
closely related to Rota-Baxter operators, such as twisted Rota-Baxter
operators~\cite{Das21}, Reynolds operators~\cite{Rey1895}, and modified
Rota--Baxter operators~\cite{Uch08}, cannot be realized as deformation maps in
the setting of twilled Lie algebras.
To overcome this limitation, Jiang, Sheng, and Tang~\cite{JST23} introduced the
notion of a  {quasi-twilled Lie algebra} together with two types of
deformation maps, namely  {type-I} and  {type-II deformation maps}.
They also developed a corresponding cohomology theory that unifies the
cohomological frameworks associated with all the operators mentioned above.
Subsequently, deformation maps for quasi-twilled associative algebras were
investigated in~\cite{Makhlouf}. The notions of deformation maps in
proto-twilled Leibniz algebras and proto-twilled Poisson algebras were studied
in~\cite{das one,das_two}. More recently, deformation maps of quasi-twilled
$3$-Lie algebras have been explored in~\cite{def_map_3_Lie}.

\medskip
%\noindent\textbf{Main contributions.}
The aim of this  paper is to introduce and  study two families of
operators on quasi-twilled Lie conformal  algebras.
A \emph{right deformation map} on quasi-twilled Lie conformal  algebra  is a $\Rg$-module homomorphism
$D:\Vs\to\Ws$ whose graph $\mathrm{Gr}(D)\subset\Ep$ is a Lie conformal 
subalgebra, or equivalently, whose defining
equation~\eqref{eq:DmapI} holds. Special cases include conformal
modified $r$-matrices (Example~\ref{ex:modified-r}), conformal crossed
homomorphisms (Example~\ref{ex:crossed-hom}), conformal derivations
(Example~\ref{ex:conf-derivation}), and Lie conformal  algebra
homomorphisms (Example~\ref{ex:conf-hom}). A \emph{left deformation
map} is a $\Rg$-module homomorphism $B:\Ws\to\Vs$
satisfying~\eqref{eq:DmapII}, and its special cases include conformal
relative Rota-Baxter operators (Example~\ref{ex:conf-RRB}), conformal
twisted Rota-Baxter operators (Example~\ref{ex:conf-TRB}), conformal
Reynolds operators (Example~\ref{ex:conf-Reynolds}), and deformation
maps of matched pairs (Example~\ref{ex:conf-matched-pair}).
For each family, we establish the following results.
\begin{enumerate}
  \item \textbf{Induced structures and Cohomology.} 
    Each deformation map $D$ (resp.\ $B$) induces a  Lie conformal
    algebra structure on $\Vs$ (resp.\ $\Ws$) and a module structure
    on the opposite summand, yielding a Chevalley-Eilenberg-type
    cochain complex whose cohomology is called the
    \emph{cohomology of the deformation map}
    (Definitions~\ref{def:cohom-I} and~\ref{def:cohom-II}).
    We show that the twisted differential $l_1^D$ (resp.\ $l_1^B$)
    coincides, up to sign, with this Chevalley-Eilenberg coboundary. 
    %(Propositions~\ref{prop:l1D-CE-I} and~\ref{prop:l1B-CE-II}),
    Moreover, we show that the second cohomology group classifies infinitesimal
    deformations.  
  \item \textbf{Controlling homotopy algebras.}
    Using Voronov's curved $V$-data construction~\cite{Vo}, we show
    that right deformation maps are Maurer-Cartan elements of a curved
    $L_\infty$-algebra on
    $\bigoplus_{n\geq0}\Cch^{n+1}(\Vs,\Ws)$
    (Theorem~\ref{thm:control-I}), and left deformation maps are
    Maurer-Cartan elements of an (uncurved) $L_\infty$-algebra on
    $\bigoplus_{n\geq0}\Cch^{n+1}(\Ws,\Vs)$
    (Theorem~\ref{thm:control-II}). In each case, the structure maps
    are given explicitly by iterated Nijenhuis-Richardson brackets
    with the components of $\str$.
  \item \textbf{Deformation theory.}
    Given a deformation map $D$ (resp.\ $B$), the twisted
    $L_\infty$-algebra governs perturbations $D+D'$ (resp.\ $B+B'$),
    and the classical Maurer-Cartan formalism yields the precise
    obstruction-theoretic description of such deformations
    (Theorems~\ref{thm:deform-DmapI} and~\ref{thm:deform-DmapII}).
\end{enumerate}

As corollaries of the general theory, we recover and extend the
controlling algebras and deformation cohomologies for conformal
$\mathcal{O}$-operators~\cite{YL},
conformal modified $r$-matrices~\cite{JS2}, conformal
derivations~\cite{TFS}, and deformation maps of matched pairs, several  
of which appear to be new even in the conformal setting.

\medskip
%\noindent\textbf{Organization.}
The paper is organized as follows. In Section~\ref{sec:prelim}, we provide a review of  Lie conformal algebras, their representations 
and cohomology, the Nijenhuis-Richardson bracket and its bigrading, and
the theory of quasi-twilled Lie conformal  algebras, culminating in the
$L_\infty$-algebra of Theorem~\ref{thm:Linf-quasi}.
In Section~\ref{sec:typeI}, we  introduce right deformation maps, construct
their controlling curved $L_\infty$-algebra, develop the deformation
theory, and define their cohomology with illustrative examples.
The last Section deals with a parallel procedure for left
deformation maps.

\medskip
\noindent\textbf{Conventions.}
Throughout this paper, all vector spaces, linear maps, and tensor products are over complex numbers field
$\C$, and all $\Rg$-modules are assumed finitely generated.

%==========================================================================
\section{Preliminaries}\label{sec:prelim}
%==========================================================================

The goal of this section is twofold. We first summarize the basics on  Lie conformal algebras, their modules, and their cohomology
in a self-contained way (see \cite{D'AndreaKac, KAC, KAC2, DK113}). Then, we give the algebraic background about bigraded cochain complexes, the Nijenhuis-Richardson bracket, and
(quasi-)twilled structures, that is needed in the sequel (For more details see \cite{Wu, YL, DK3}).

%--------------------------------------------------------------------------
\subsection{Lie conformal algebras}\label{subsec:conf-alg}
%--------------------------------------------------------------------------

\begin{definition}\label{def:conformal-algebra}
A \textit{conformal algebra} is a $\Rg$-module $\Ep$ endowed with a
$\C$-linear map
\[
  \mu : \Ep \otimes \Ep \;\longrightarrow\; \Ep[\lambda],
  \qquad a \otimes b \;\mapsto\; a_\lambda b,
\]
satisfying \emph{conformal sesquilinearity}:
\begin{equation}\label{eq:sesqui}
  (\partial a)_\lambda b = -\lambda\,a_\lambda b,
  \qquad
  a_\lambda(\partial b) = (\partial + \lambda)\,a_\lambda b,
  \qquad \forall\; a, b \in \Ep.
\end{equation}
\end{definition}

From \eqref{eq:sesqui} one derives, by variable substitution, the
secondary identities ($a,b,c\in\Ep$):
\begin{align}
  (a_{-\la-\p}b)_{\la+\mu}c
  &= (a_\mu b)_{\la+\mu}c, &
  a_\mu(b_{-\la-\p}c)
  &= a_\mu(b_{-\la-\mu-\p}c), \label{eq:sesqui-extra1} \\
  (a_{-\mu-\p}b)_{-\la-\p}c
  &= (a_{-\la-\mu-\p}b)_{-\la-\p}c, &
  a_{-\la-\mu-\p}(b_{-\la-\p}c)
  &= a_{-\la-\mu-\p}(b_\mu c). \label{eq:sesqui-extra2}
\end{align}
These identities will be used repeatedly in later computations.

%--------------------------------------------------------------------------
%\subsection{Lie conformal  algebras}\label{subsec:Lie-conf}
%--------------------------------------------------------------------------

\begin{definition}\label{def:Lie-conf}
A \textit{Lie conformal  algebra} is a conformal algebra  $\Ep$  with product 
$\Ep \otimes \Ep \to \Ep[\lambda]$, $a\otimes b \mapsto \{a_\lambda b\}$,
called the \emph{$\lambda$-bracket}, subject to the following identities  for
all $a,b,c\in\Ep$:
\begin{align}
    % \label{Sesquilinearity}
   %\{\partial a_\lambda b\} = -\lambda\{a_\lambda b\}&\text{ and %}
    %\{a_\lambda \partial b\} = (\partial+\lambda)\{a_\lambda b\}, \text{  (Sesquilinearity) }\\
  \label{Skew-symmetry}
    \{a_\lambda b\} &= -\{b_{-\lambda-\partial}a\},\quad\quad\quad\quad\quad\;\; \;\;\text{ (Skew-symmetry)}\\
  \label{Jacobi identity}
    \{a_\lambda\{b_\mu c\}\}
    & = \{\{a_\lambda b\}_{\lambda+\mu}c\} + \{b_\mu\{a_\lambda c\}\}, \text{ (Jacobi identity)} .
\end{align}
\end{definition}

%--------------------------------------------------------------------------

Let $\Mod$ and $\Mod'$ be $\Rg$-modules. A \textit{conformal linear
map} from $\Mod$ to $\Mod'$ is a $\C$-linear map
$f_\lambda : \Mod \to \Mod'[\lambda]$ satisfying
\[
  f_\lambda(\partial u) = (\partial + \lambda)f_\lambda u,
  \qquad \forall\; u \in \Mod.
\]
The $\C$-vector space of all such maps is denoted $\mathrm{Chom}(\Mod,\Mod')$
and is itself a $\Rg$-module via $(\partial f)_\lambda := -\lambda f_\lambda$.
The \textit{conformal dual} of $\Mod$ is the $\Rg$-module
\[
  \Mod^{*c} := \mathrm{Chom}(\Mod,\C),
\]
where $\C$ is considered as the trivial $\Rg$-module ($\partial\cdot 1 = 0$).
Concretely,
$\Mod^{*c} = \{ \xi: \Mod\to\C[\lambda] \mid \xi_\lambda(\partial u)
= \lambda\,\xi_\lambda u, \text{ for all } u\in\Mod\}.$

For a finitely generated $\Rg$-module $\Mod$, the space
$\mathrm{Cend}(\Mod) := \mathrm{Chom}(\Mod,\Mod)$
is an associative conformal algebra under the composition product
$(f_\lambda g)_\mu v := f_\lambda(g_{\mu-\lambda}v)$.
The associated Lie conformal  algebra, denoted $\mathrm{gc}(\Mod)$
and called the \textit{general Lie conformal  algebra} on $\Mod$, has
brackets
\[
  [f_\lambda g]_\mu v
  := f_\lambda(g_{\mu-\lambda}v) - g_{\mu-\lambda}(f_\lambda v),
  \qquad f,g \in \mathrm{Cend}(\Mod),\; v\in\Mod.
\]

%--------------------------------------------------------------------------
%\subsection{Modules over Lie conformal  algebras and cohomology}\label{subsec:modules}
%--------------------------------------------------------------------------

\begin{definition}\label{def:module}
A $\Rg$-module $\Mod$ is a \textit{module} over a Lie conformal  algebra
$(\Ep,\{\cdot_\lambda\cdot\})$ if there exists a $\C$-linear map
$\rep:\Ep\to\mathrm{Cend}(\Mod)$ such that
\begin{equation}\label{eq:module}
  \rep(a)_\lambda\rep(b)_\mu - \rep(b)_\mu\rep(a)_\lambda
  = \rep(\{a_\lambda b\})_{\lambda+\mu},
  \qquad
  \rep(\partial a)_\lambda = -\lambda\rep(a)_\lambda,
  \quad \forall\; a,b\in\Ep.
\end{equation}
Equivalently, $\rep$ is a homomorphism of Lie conformal  algebras
$\Ep\to\mathrm{gc}(\Mod)$. We write $(\Mod;\rep)$ for this module.
\end{definition}

One verifies from Definition~\ref{def:module} the supplementary
identities
\begin{align}
  \rep(\{a_\lambda b\})_{-\p-\mu}
  &= \rep(a)_\lambda\rep(b)_{-\p-\mu}
     - \rep(b)_{-\p-\la-\mu}\rep(a)_\lambda, \label{eq:mod-extra1}\\
  \rep(\{a_\mu b\})_{-\p-\la}
  &= \rep(a)_\mu\rep(b)_{-\p-\la}
     - \rep(b)_{-\p-\la-\mu}\rep(a)_{-\p-\la}, \label{eq:mod-extra2}
\end{align}
which appear in the explicit computation of coboundary operators.

The \textit{dual module} structure on $\Mod^{*c}$ is given by
$\rep^*:\Ep\to\mathrm{gc}(\Mod^{*c})$,
\[
  (\rep^*(a)_\lambda \xi)_\mu v
  := -\xi_{\mu-\lambda}(\rep(a)_\lambda v),
  \qquad a\in\Ep,\; \xi\in\Mod^{*c},\; v\in\Mod.
\]
Setting $\ad(a)_\lambda b := \{a_\lambda b\}$ for $a,b\in\Ep$ yields the
\textit{adjoint module} $(\Ep;\ad)$, and hence the
\textit{coadjoint module} $(\Ep^{*c};\ad^*)$.

\begin{proposition}[\cite{DK113}]
\label{prop:semidirect}
Let $\Ep$ be a Lie conformal  algebra and $(\Mod;\rep)$ be a module.  The
$\Rg$-module $\Ep\oplus\Mod$ carries a Lie conformal  algebra structure
\begin{equation}\label{eq:semidirect}
  [(a,m)_\lambda(b,n)]
  := \bigl(\{a_\lambda b\},\;
            \rep(a)_\lambda n - \rep(b)_{-\p-\lambda}m\bigr),
  \quad a,b\in\Ep,\; m,n\in\Mod,
\end{equation}
called the \textit{semi-direct product} and denoted $\Ep\ltimes_\rep\Mod$.
\end{proposition}

More generally, given a $2$-cocycle $\phi\in\Cch^2(\Ep,\Mod)$, the formula
\begin{equation}\label{eq:twisted-semidirect}
  [(a,m)_\lambda(b,n)]^\phi
  := \bigl(\{a_\lambda b\},\;
            \rep(a)_\lambda n - \rep(b)_{-\p-\lambda}m + \phi_\lambda(a,b)\bigr)
\end{equation}
defines a Lie conformal  algebra structure on $\Ep\oplus\Mod$, called the
\textit{$\phi$-deformed semi-direct product} and denoted
$\Ep\ltimes_\phi\Mod$.

Let us recall the cohomology complex for a Lie conformal algebra with coefficients in a given 
module introduced in \cite{DK113}. Fix a Lie conformal  algebra $(\Ep,\{\cdot_\lambda\cdot\})$ and a module
$(\Mod;\rep)$.  Set $\Cch^0(\Ep,\Mod) := \Mod/\partial\Mod$.  For
$k\geq 1$, let $\Cch^k(\Ep,\Mod)$ be the space of $\C$-linear maps
\[
  \varphi : \Ep^{\otimes k} \longrightarrow
  \C[\lambda_1,\ldots,\lambda_{k-1}]\otimes\Mod
\]
satisfying the following two conditions, where
$\lambda_k^\ddag := -\sum_{j=1}^{k-1}\lambda_j - \partial^\Mod$:

\smallskip
\noindent\textit{(C1) Sesquilinearity.}
For $1\leq i\leq k-1$,
$$\varphi_{\boldsymbol\lambda}(\ldots,\partial a_i,\ldots)
= -\lambda_i\,\varphi_{\boldsymbol\lambda}(\ldots,a_i,\ldots),
\text{ and }
\varphi_{\boldsymbol\lambda}(\ldots,\partial a_k)
= -\lambda_k^\ddag\,\varphi_{\boldsymbol\lambda}(\ldots,a_k).$$

\smallskip
\noindent\textit{(C2) Graded antisymmetry.}
For every permutation $\sigma$ of $\{1,\ldots,k\}$,
\begin{equation}\label{eq:antisym}
  \varphi_{\boldsymbol\lambda}(a_1,\ldots,a_k)
  = (-1)^\sigma\,
    \varphi_{\lambda_{\sigma(1)},\ldots,\lambda_{\sigma(k-1)}}
    (a_{\sigma(1)},\ldots,a_{\sigma(k)})
    \big|_{\lambda_k\mapsto\lambda_k^\ddag}.
\end{equation}

\medskip
The \textit{coboundary operator}
$\dM:\Cch^k(\Ep,\Mod)\to\Cch^{k+1}(\Ep,\Mod)$
is defined for $\varphi\in\Cch^k(\Ep,\Mod)$, $k\geq 1$, by
\begin{equation}\label{eq:coboundary}
\begin{split}
  (\dM\varphi)_{\la_1,\ldots,\la_k}
  &(a_1,\ldots,a_{k+1}) \\
  &= \sum_{i=1}^{k}(-1)^{i+1}
     \rep(a_i)_{\la_i}
     \varphi_{\la_1,\ldots,\widehat{\la_i},\ldots \la_k}(a_1,\ldots,\widehat{a_i},\ldots,a_{k+1})
     %\big|_{\la_{k+1}\mapsto\la_{k+1}^\ddag} 
     \\
  &\quad
    + \sum_{i=1}^{k}(-1)^i
    \varphi_{\la_1,\ldots,\widehat{\la_i},\ldots \la_k}
      (a_1,\ldots,\widehat{a_i},\ldots,a_k,
       \{a_{i\,\la_i} a_{k+1}\})
      %\big|_{\la_{k+1}\mapsto\la_{k+1}^\ddag} 
      \\
  &\quad
    + \!\!\sum_{1\leq i<j\leq k+1}\!\!(-1)^{k+i+j+1}
      \varphi_{\la_1,\ldots,\widehat{\la_i},\ldots ,\widehat{\la_j},\ldots, \la_k,\la_{k+1}^\ddag}
      (a_1,\ldots,a_k,\{a_{i\,\la_i}a_j\})
      %\big|_{\la_{k+1}\mapsto\la_{k+1}^\ddag} 
      \\
  &\quad
    + (-1)^k\rep(a_{k+1})_{\la_{k+1}^\ddag}
      \varphi_{\la_1,\ldots,\la_{k-1}}(a_1,\ldots,a_k),
\end{split}
\end{equation}
where hats denote the omission of the corresponding index or argument. For $\bar m \in \Cch^0(\Ep,\Mod)$ we have $(\dM \bar m)(a)=\varrho(a)_{\partial^{-\mathcal M}}m$.

\begin{theorem}[{\cite{DK113}}]\label{thm:d-squared}
  For every $\varphi\in\Cch^k(\Ep,\Mod)$, one has $\dM\varphi\in\Cch^{k+1}(\Ep,\Mod)$
  and $\dM^2\varphi = 0$.  Hence $(\Cch^*(\Ep,\Mod),\dM)$ is a cochain complex,
  whose cohomology $\mathrm{H}^*(\Ep,\Mod)$ is the
  \textit{Lie conformal  algebra cohomology} of $\Ep$ with values in $(\Mod;\rep)$.
\end{theorem}

%--------------------------------------------------------------------------
\subsection{The Nijenhuis-Richardson bracket and bigrading}
\label{subsec:FN}
%--------------------------------------------------------------------------

We now enrich the cochain complex $\Cch^*(\Ep,\Ep)$ with an additional
algebraic structure,  a graded Lie bracket  that underlies all the
operator-theoretic constructions in this paper.

A permutation $\sigma\in\perm_n$ is an \textit{$(i,n-i)$-unshuffle}
if $\sigma(1)<\cdots<\sigma(i)$ and $\sigma(i+1)<\cdots<\sigma(n)$; the
set of all such permutations is denoted $\perm_{(i,n-i)}$.

Set $\Cch^*(\Ep,\Ep) := \bigoplus_{k\geq1}\Cch^k(\Ep,\Ep)$.  For
$f\in\Cch^m(\Ep,\Ep)$ and $g\in\Cch^n(\Ep,\Ep)$, define the
\textit{Nijenhuis-Richardson 
%(FN)
(NR) bracket}
\begin{equation}\label{eq:FN-bracket}
  [f,g]_{NR} := f \star g - (-1)^{(m-1)(n-1)}\,g\star f,
\end{equation}
where the \emph{pre-bracket} $(f\star g)\in\Cch^{m+n-1}(\Ep,\Ep)$ is
\begin{align*}
  &(f\star g)_{\la_1,\ldots,\la_{m+n-2}}(a_1,\ldots,a_{m+n-1}) \\
  =&  \sum_{\sigma\in\perm_{(n,m-1)}}(-1)^\sigma
    f_{\la_{\sigma(1)}+\cdots+\la_{\sigma(n)},\,
       \la_{\sigma(n+1)},\ldots,\la_{\sigma(m+n-2)}}\\\quad\quad\quad\quad\quad\quad\quad\quad\quad&
    \bigl(g_{\la_{\sigma(1)},\ldots,\la_{\sigma(n-1)}}
          (a_{\sigma(1)},\ldots,a_{\sigma(n)}),
    a_{\sigma(n+1)},\ldots,a_{\sigma(m+n-1)}\bigr)
    \big|_{\la_{m+n-1}\mapsto\la_{m+n-1}^\ddag},
\end{align*}
with $\la_{m+n-1}^\ddag := -\sum_{i=1}^{m+n-2}\la_i - \partial$.

\begin{lemma}[{\cite{DK13}}]\label{lem:FN-gLa}
 The pair  $(\Cch^*(\Ep,\Ep),[-,-]_{NR})$ defines a graded Lie algebra.  A $2$-cochain
  $\omega\in\Cch^2(\Ep,\Ep)$ defines a Lie conformal  algebra structure
  on $\Ep$ via $\{a_\lambda b\}:=\omega_\lambda(a,b)$ if and only if
  $[\omega,\omega]_{NR} = 0$.
\end{lemma}

Let $\Vs$ and $\Ws$ be $\Rg$-modules (not yet assumed to carry any
algebraic structure). Elements of $\Vs$ will be written $a,b,a_i,\ldots$
and elements of $\Ws$ as $u,v,u_i,\ldots$.

\begin{definition}\label{def:bidegree}
A cochain $f\in\Cch^{k+l+1}(\Vs\oplus\Ws,\Vs\oplus\Ws)$ has
\textit{bidegree} $k\mid l$ (written $\|f\|=k\mid l$) if the following
three conditions hold:
\begin{itemize}
  \item[\rm(i)] If $X\in\Vs^{\otimes k+1}\otimes\Ws^{\otimes l}$,
    then $f(X)\in\C[\la_1,\ldots,\la_{k+l}]\otimes\Vs$.
  \item[\rm(ii)] If $X\in\Vs^{\otimes k}\otimes\Ws^{\otimes l+1}$,
    then $f(X)\in\C[\la_1,\ldots,\la_{k+l}]\otimes\Ws$.
  \item[\rm(iii)] $f(X)=0$ in all remaining input-type combinations.
\end{itemize}
We call $f$ \textit{bihomogeneous} when it has a bidegree.
\end{definition}

\begin{lemma}\label{lem:FN-bigrade}
  If $\|f\|=k_f\mid l_f$ and $\|g\|=k_g\mid l_g$, then
  $[f,g]_{NR}$ has bidegree $k_f+k_g\mid l_f+l_g$.
\end{lemma}

\begin{lemma}\label{lem:FN-zero}
  If $\|f\|=-1\mid l$ (resp.\ $l\mid-1$) and $\|g\|=-1\mid k$
  (resp.\ $k\mid-1$), then $[f,g]_{NR} = 0$.
\end{lemma}

Given $\alpha:\Vs\otimes\Vs\to\Vs[\lambda]$ and
$\beta:\Vs\otimes\Ws\to\Ws[\lambda]$, their \textit{lifts} to
$\Cch^2(\Vs\oplus\Ws,\Vs\oplus\Ws)$ are defined by
\begin{align}
  \widehat\alpha_\la\bigl((a_1,u_1),(a_2,u_2)\bigr)
  &:= \bigl(\alpha_\la(a_1,a_2),\,0\bigr), \label{eq:lift-a}\\
  \widehat\beta_\la\bigl((a_1,u_1),(a_2,u_2)\bigr)
  &:= \bigl(0,\,\beta_\la(a_1,u_2)-\beta_{-\p-\la}(a_2,u_1)\bigr),
  \label{eq:lift-b}
\end{align}
and both have bidegree $1\mid 0$.  Setting $\widehat\nu := \widehat\alpha+\widehat\beta$,
we obtain the semi-direct product type operation
\begin{equation}\label{eq:nu-hat}
  \widehat\nu_\la\bigl((a_1,u_1),(a_2,u_2)\bigr)
  = \bigl(\alpha_\la(a_1,a_2),\;
           \beta_\la(a_1,u_2)-\beta_{-\la-\p}(a_2,u_1)\bigr).
\end{equation}

The following result recasts the module axiom in terms of the 
%FN 
NR bracket
and supplies the cohomological interpretation of $\dM$.

\begin{proposition}\label{prop:FN-module}
  Let $(\Ep,\omega)$ be a Lie conformal  algebra and
  $\rep:\Ep\to\mathrm{Cend}(\Mod)$ a map with
  $\rep(\partial a)_\lambda=-\lambda\rep(a)_\lambda$.  Then
  $(\Mod;\rep)$ is a module if and only if
  $[\widehat\omega+\widehat\rep,\;\widehat\omega+\widehat\rep]_{NR} = 0$.
  Moreover, the coboundary operator \eqref{eq:coboundary} satisfies
  \begin{equation}\label{eq:FN-coboundary}
    \dM_{\omega+\rep}\,\varphi
    = (-1)^{k-1}[\widehat\omega+\widehat\rep,\;\widehat\varphi]_{NR},
    \quad \varphi\in\Cch^k(\Ep,\Mod).
  \end{equation}
\end{proposition}

%--------------------------------------------------------------------------
\subsection{Twilled and quasi-twilled  Lie conformal algebras}
\label{subsec:twilled}
%--------------------------------------------------------------------------

Let $(\Ep,\{\cdot_\lambda\cdot\})$ be a Lie conformal algebra with a
fixed direct-sum decomposition $\Ep = \Vs\oplus\Ws$ of $\Rg$-modules.
Denote by $\mathrm{pr}_\Vs:\Ep\to\Vs$ and $\mathrm{pr}_\Ws:\Ep\to\Ws$
the natural projections, and define six structural maps by
\begin{align*}
  \{a_\la b\}_\Vs
    &:= \mathrm{pr}_\Vs(\{a_\la b\}), &
  \rep_\Ws(u)_\la a
    &:= \mathrm{pr}_\Vs(\{u_\la a\}), &
  \psi_{2,\la}(u,v)
    &:= \mathrm{pr}_\Vs(\{u_\la v\}), \\
  \{u_\la v\}_\Ws
    &:= \mathrm{pr}_\Ws(\{u_\la v\}), &
  \rep_\Vs(a)_\la v
    &:= \mathrm{pr}_\Ws(\{a_\la v\}), &
  \psi_{1,\la}(a,b)
    &:= \mathrm{pr}_\Ws(\{a_\la b\}),
\end{align*}
for $a,b\in\Vs$ and $u,v\in\Ws$.  The full $\lambda$-bracket on $\Ep$ reads
{\small\begin{equation}\label{eq:full-bracket}
  \{(a,u)_\la(b,v)\}
  = \bigl(\{a_\la b\}_\Vs + \rep_\Ws(u)_\la b
          - \rep_\Ws(v)_{-\la-\p}a + \psi_{2,\la}(u,v),\;
          \{u_\la v\}_\Ws + \rep_\Vs(a)_\la v
          - \rep_\Vs(b)_{-\la-\p}u + \psi_{1,\la}(a,b)\bigr).
\end{equation}}

Denoting the Lie conformal algebra structure on $\Ep$ by $\str\in\Cch^2(\Ep,\Ep)$,
Lemma~\ref{lem:FN-bigrade} yields the unique bihomogeneous decomposition
\begin{equation}\label{eq:str-decomp}
  \str = \shf + \smu + \smv + \widehat\psi_2,
\end{equation}
where $\|\shf\|=2\mid{-1}$, $\|\smu\|=1\mid 0$, $\|\smv\|=0\mid 1$,
$\|\widehat\psi_2\|=-1\mid 2$, with explicit formulas
\begin{align}
  \shf_\la((a,u),(b,v))
    &= (0,\,\psi_{1,\la}(a,b)),
    \label{eq:comp-psi1}\\
  \widehat{\nu}_{1,\la}((a,u),(b,v))
    &= (\{a_\la b\}_\Vs,\;
        \rep_\Vs(a)_\la v - \rep_\Vs(b)_{-\la-\p}u),
    \label{eq:comp-nu1}\\
  \widehat{\nu}_{2,\la}((a,u),(b,v))
    &= (\rep_\Ws(u)_\la b - \rep_\Ws(v)_{-\la-\p}a,\;
        \{u_\la v\}_\Ws),
    \label{eq:comp-nu2}\\
  \widehat\psi_{2,\la}((a,u),(b,v))
    &= (\psi_{2,\la}(u,v),\,0).
    \label{eq:comp-psi2}
\end{align}

The Maurer-Cartan equation $[\str,\str]_{NR}=0$ decomposes by bidegree into
\begin{equation}\label{eq:MC-system}
  \begin{cases}
    [\smu,\shf]_{NR} = 0, \\[3pt]
    \half[\smu,\smu]_{NR} + [\smv,\shf]_{NR} = 0, \\[3pt]
    [\smu,\smv]_{NR} + [\shf,\widehat\psi_2]_{NR} = 0, \\[3pt]
    \half[\smv,\smv]_{NR} + [\smu,\widehat\psi_2]_{NR} = 0, \\[3pt]
    [\smv,\widehat\psi_2]_{NR} = 0.
  \end{cases}
\end{equation}

\begin{definition}\label{def:twilled}
Let $\Ep=\Vs\oplus\Ws$ be a Lie conformal algebra with structure
\eqref{eq:str-decomp}.
\begin{itemize}
  \item[\rm(i)] The triple $(\Ep,\Vs,\Ws)$ is called a
    \textit{quasi-twilled Lie conformal  algebra} if $\psi_2=0$,
    i.e.\ $\Ws$ is a sub-algebra of $\Ep$. The structure then reduces to
    $\str=\shf_1+\smu+\smv$.
  \item[\rm(ii)] The triple $(\Ep,\Vs,\Ws)$ is called an
    \textit{twilled Lie conformal  algebra} (or \textit{bicrossed
    product}) if $\psi_1=\psi_2=0$, i.e.\ both $\Vs$ and $\Ws$ are
    sub-algebras of $\Ep$.  We then write $\Ep=\Vs\bowtie\Ws$.
\end{itemize}
\end{definition}

When $(\Ep,\Vs,\Ws)$ is quasi-twilled, Equation~\eqref{eq:MC-system}
reduces to the four conditions
\begin{equation}\label{eq:quasi-MC}
  [\smu,\shf]_{NR} = 0,\quad
  \half[\smu,\smu]_{NR} + [\smv,\shf]_{NR} = 0,\quad
  [\smu,\smv]_{NR} = 0,\quad
  \half[\smv,\smv]_{NR} = 0.
\end{equation}
In particular, $(\Ws,\{\cdot_\la\cdot\}_\Ws)$ is a Lie conformal  algebra
and $(\Vs;\rep_\Ws)$ is a module over it, while $(\Vs,\{\cdot_\la\cdot\}_\Vs)$
need not be a Lie conformal  algebra in general (it is so precisely in
the twilled case).

%--------------------------------------------------------------------------
\subsection{$L_\infty$-algebras and Maurer--Cartan elements}
\label{subsec:Linf}
%--------------------------------------------------------------------------

\begin{definition}[{\cite{LS,LM}}]\label{def:Linf}
An \textit{$L_\infty$-algebra} is a graded vector space
$\mathfrak{g}=\bigoplus_{i\in\Nat}\mathfrak{g}^i$ equipped with
multilinear maps $\ell_k:\otimes^k\mathfrak{g}\to\mathfrak{g}$ of
degree $2-k$ ($k\geq 1$) satisfying, for all homogeneous
$v_1,\ldots,v_n\in\mathfrak{g}$:
\begin{itemize}
  \item[\rm(i)] \emph{Graded antisymmetry:}
    $\ell_k(v_{\sigma(1)},\ldots,v_{\sigma(k)})
    = \chi(\sigma)\,\ell_k(v_1,\ldots,v_k)$
    for every $\sigma\in\perm_k$, where $\chi(\sigma)$ is the Koszul sign.
  \item[\rm(ii)] \emph{Higher Jacobi identity:} for each $n\geq1$,
    \[
      \sum_{i+j=n+1}(-1)^i
      \sum_{\sigma\in\perm_{(i,n-i)}}
      \!\!\chi(\sigma)\;\ell_j\!\bigl(\ell_i(v_{\sigma(1)},\ldots,v_{\sigma(i)}),
      v_{\sigma(i+1)},\ldots,v_{\sigma(n)}\bigr) = 0.
    \]
\end{itemize}
An element $\alpha\in\mathfrak{g}^1$ is a \textit{Maurer-Cartan element} if
\[
  \sum_{k=1}^\infty \frac{1}{k!}\,\ell_k(\alpha,\ldots,\alpha) = 0.
\]
\end{definition}

The following theorem together with \cite[Corollary~3.5]{Uch1}
supplies the $L_\infty$-algebra that governs all subsequent deformation
problems.

\begin{theorem}\label{thm:Linf-quasi}
  Let $(\Ep,\Vs,\Ws,\str)$ be a quasi-twilled Lie conformal  algebra
  with $\str = \shf+\smu+\smv$.  Define maps on $\Cch^*(\Ws,\Vs)$ by
  \begin{align}
    \delta(f)
    &:= [\smv,\,\widehat f]_{NR}, \label{eq:l1}\\
    \{f_1,f_2\}
    &:= (-1)^{m-1}[[\smu,\widehat f_1]_{NR},\,\widehat f_2]_{NR},
    \label{eq:l2}\\
    \{f_1,f_2,f_3\}
    &:= (-1)^{n-1}[[[\shf,\widehat f_1]_{NR},\,\widehat f_2]_{NR},
       \,\widehat f_3]_{NR},\label{eq:l3}
  \end{align}
  for all $f\in\Cch^*(\Ws,\Vs)$, $f_1\in\Cch^m(\Ws,\Vs)$,
  $f_2\in\Cch^n(\Ws,\Vs)$, $f_3\in\Cch^k(\Ws,\Vs)$.  Then
  $(\Cch^*(\Ws,\Vs),\,\delta,\,\{\cdot,\cdot\},\,\{\cdot,\cdot,\cdot\})$
  is an $L_\infty$-algebra.
\end{theorem}

\begin{corollary}\label{cor:dgla-twilled}
  For a twilled Lie conformal  algebra $\Vs\bowtie\Ws$
   in which $\shf=0$ and \eqref{eq:l3} vanishes,
  $(\Cch^*(\Ws,\Vs),\,\delta,\,\{\cdot,\cdot\})$ is a differential
  graded Lie algebra.
\end{corollary}

%=======================================================
%==========================================================================

%==========================================================================
\section{Controlling algebras and cohomologies of right deformation maps}
\label{sec:typeI}
%==========================================================================
 
In this section, we introduce the notion of right deformation map, which unify
 modified $r$-matrices,  crossed homomorphisms,
 derivations, and  Lie conformal algebra homomorphisms. The triple 
$(\Ep,\Vs,\Ws)$ is always a quasi-twilled Lie conformal algebra with structure $\str=\shf+\smu+\smv$, where the six components are
\[
  \{a_\la b\}_\Vs,\quad
  \rep_\Vs(a)_\la v,\quad
  \{u_\la v\}_\Ws,\quad
  \rep_\Ws(u)_\la a,\quad
  \psi_{\la}(a,b),
\]
for all $a,b\in\Vs$ and $u,v\in\Ws$ (and $\psi_2=0$).
 
%--------------------------------------------------------------------------
\subsection{Right deformation maps}
\label{subsec:DmapI}
%--------------------------------------------------------------------------

Let $D:\Vs\to\Ws$ be a $\Rg$-module homomorphism.  Its \textit{graph} is the $\Rg$-submodule
\[
  \mathrm{Gr}(D) := \{(a,D(a))\mid a\in\Vs\} \subset \Vs\oplus\Ws = \Ep.
\]

\begin{definition}\label{def:DmapI}
Let $(\Ep,\Vs,\Ws)$ be a quasi-twilled Lie conformal algebra.
A right deformation map of $(\Ep,\Vs,\Ws)$
is a $\Rg$-module homomorphism $D:\Vs\longrightarrow\Ws$
(i.e.\ a $\C$-linear map satisfying $D(\p a)=\p D(a)$ for all $a\in\Vs$)
such that $\mathrm{Gr}(D)$ is a  Lie conformal sub-algebra of $\Ep$.
\end{definition}

\begin{proposition}\label{prop:graph-DmapI}
A $\Rg$-module homomorphism $D:\Vs\to\Ws$ is a right deformation map  if and only if, for all $a,b\in\Vs$, we have
\begin{align}
&  D\!\bigl(\rep_\Ws(D(a))_\la b - \rep_\Ws(D(b))_{-\la-\p}a
           + \{a_\la b\}_\Vs\bigr) \nonumber \\
  =&  \{D(a)_\la D(b)\}_\Ws
    + \rep_\Vs(a)_\la D(b) - \rep_\Vs(b)_{-\la-\p}D(a)
    + \psi_{1,\la}(a,b). \label{eq:DmapI}
\end{align}

\end{proposition}
 
\begin{proof}
For $(a,D(a)),(b,D(b))\in\mathrm{Gr}(D)$, the $\la$-bracket of $\Ep$ gives
\begin{align*}
  \str_\la\!\bigl((a,D(a)),(b,D(b))\bigr)
  = \Bigl(
    &\{a_\la b\}_\Vs + \rep_\Ws(D(a))_\la b
       - \rep_\Ws(D(b))_{-\la-\p}a,\\
    &\{D(a)_\la D(b)\}_\Ws + \rep_\Vs(a)_\la D(b)
       - \rep_\Vs(b)_{-\la-\p}D(a) + \psi_{1,\la}(a,b)
  \Bigr).
\end{align*}
Then $D:\Vs\to\Ws$ is a right deformation map if and only if the $\Ws$-component equals
$D$ applied to the $\Vs$-component, which is precisely Eq. \eqref{eq:DmapI}.
\end{proof}

\begin{remark}\label{rmk:module-hom}
We require $D$ to be a $\Rg$-\emph{module} homomorphism, that is, a
$\C$-linear map satisfying $D(\p a)=\p D(a)$.  This condition is
strictly weaker than conformal linearity, which would demand a
$\la$-parametric family $D_\la$ with $D_\la(\p a)=(\p+\la)D_\la a$.
The module-homomorphism condition is the natural one for deformation
maps: it ensures that $D$ is compatible with the $\Rg$-module structure
of $\Vs$ and $\Ws$ without introducing a spectral-parameter shift, and
it is precisely the condition needed for the graph $\mathrm{Gr}(D)$ to
be a $\Rg$-submodule of $\Ep=\Vs\oplus\Ws$.
\end{remark}
 
\begin{remark}\label{rmk:DmapI-existence}
 Consider the $\phi$-deformed semi-direct product
$\Ep=\Vs\ltimes_\phi\Mod$ (with $\Ws=\Mod$, $\{u_\la v\}_\Ws=0$,
$\rep_\Ws=0$, $\psi_1=\phi\in\Cch^2(\Vs,\Mod)$ a $2$-cocycle).
A $\Rg$-module homomorphism $D:\Vs\to\Mod$ is a right deformation map\ if and only if
\[
  \phi_\la(a,b)
  = -\bigl(\rep(a)_\la D(b) - \rep(b)_{-\la-\p}D(a) - D(\{a_\la b\}_\Vs)\bigr)
  = \delta(-D)_\la(a,b)
\]
  for all \,$a,b\in\Vs$, where $\delta$ is the Chevalley-Eilenberg coboundary of $(\Vs,\{\cdot_\la\cdot\}_\Vs)$
with coefficients in $(\Mod;\rep)$.
Thus $\Ep\ltimes_\phi\Mod$ admits a right deformation map\ if and only if $\phi$ is an
exact $2$-cocycle.
\end{remark}

\begin{remark}
Proposition~\ref{prop:graph-DmapI} shows that finding a complement of
$\Ws$ in $\Ep$ that forms a matched pair with $\Ws$ is equivalent to
finding a right deformation map\ of $(\Ep,\Vs,\Ws)$.
\end{remark}
 
\begin{example}[Modified $r$-matrix]\label{ex:modified-r} Modified $r$-matrices play an important role in the study of
conformal analogues of Lax equations and factorization problems in Lie conformal  algebras, extending the classical picture of
Semenov-Tian-Shansky \cite{STS,STS2}. 
Let $\Ep=\Vs\oplus\Vs$ with $\{a_\la b\}_\Vs=\{u_\la v\}_\Ws$,
$\rep_\Vs=\rep_\Ws=\ad$, and
$\psi_{1,\la}(a,b)=-\{a_\la b\}_\Vs$.
A $\Rg$-module homomorphism $D:\Vs\to\Vs$ is a right deformation map\ if and only if
\begin{equation}\label{eq:conf-modified-r}
  \{D(a)_\la D(b)\}_\Vs 
  =D\!\bigl(\{D(a)_\la b\}_\Vs
  + \{a_\la D(b)\}_\Vs\bigr) -\la\{a_\la b\}_\Vs, \quad \forall\,a,b\in\Vs.
\end{equation}
This is the \textit{ modified Yang-Baxter equation} on
$(\Vs,\{\cdot_\la\cdot\}_\Vs)$; its solutions are called
\textit{ modified $r$-matrices}.
\end{example}
 
\begin{example}[Crossed homomorphism]\label{ex:crossed-hom}
Let $\Ep=\Vs\ltimes_\rep\Ws$ be the weighted semi-direct product and 
$\psi_1=0$.
A $\Rg$-module homomorphism $D:\Vs\to\Ws$ is a right deformation map\ if and only if
\begin{equation}\label{eq:conf-crossed-hom}
  D(\{a_\la b\}_\Vs)
  = \rep(a)_\la D(b) - \rep(b)_{-\la-\p}D(a)
  + \{D(a)_\la D(b)\}_\Ws, \quad \forall\,a,b\in\Vs,
\end{equation}
i.e.\ $D$ is a \textit{crossed homomorphism}
from $(\Vs,\{\cdot_\la\cdot\}_\Vs)$ to $(\Ws,\{\cdot_\la\cdot\}_\Ws)$.
\end{example}

\begin{example}[Derivation]\label{ex:conf-derivation}
Let $\Ep=\Vs\ltimes_\rep\Mod$ with $\{u_\la v\}_\Ws=0$, $\rep_\Ws=0$, $\psi_1=0$.
A $\Rg$-module homomorphism $D:\Vs\to\Mod$ is a right deformation map\ if and only if
\begin{equation}\label{eq:conf-derivation}
  D(\{a_\la b\}_\Vs)
  = \rep(a)_\la D(b) - \rep(b)_{-\la-\p}D(a),
  \quad \forall\,a,b\in\Vs,
\end{equation}
i.e.\ $D$ is a \textit{ derivation} from
$(\Vs,\{\cdot_\la\cdot\}_\Vs)$ to $(\Mod;\rep)$.  When $\rep=\ad$,
one  recovers the usual notion of a derivation on a Lie conformal algebra.
\end{example}
 
\begin{example}[Lie conformal  algebra morphism]\label{ex:conf-hom}
Let $\Ep=\Vs\oplus\Ws$ be the direct product with
$\rep_\Vs=\rep_\Ws=0$, $\psi_1=0$.
A $\Rg$-module homomorphism $D:\Vs\to\Ws$ is a right deformation map\ if and only if
\begin{equation}\label{eq:conf-Lie-hom}
  D(\{a_\la b\}_\Vs) = \{D(a)_\la D(b)\}_\Ws,
  \quad \forall\,a,b\in\Vs,
\end{equation}
i.e.\ $D$ is a \textit{ Lie algebra conformal morphism} from
$(\Vs,\{\cdot_\la\cdot\}_\Vs)$ to $(\Ws,\{\cdot_\la\cdot\}_\Ws)$.
\end{example}

%--------------------------------------------------------------------------
\subsection{Cohomology of right deformation maps}
\label{subsec:cohom-I}
%--------------------------------------------------------------------------

%======Induced structures===========
Let $(\Ep,\Vs,\Ws)$ be a quasi-twilled Lie conformal algebra and
$D:\Vs\to\Ws$ be a right deformation map. Define the map $\Bar{D}: \Ep \to \Ep,\ (a,u) \mapsto (0,D(a))$. It is  clear that $\Bar{D}^2=0$. 
Consider the map $I+\Bar{D}: \Ep \to \Ep,\ (a,u) \mapsto (a,u+D(a))$. It is an invertible $\mathbb{C}[\partial]$-homomorphism with inverse $I-\Bar{D}$. Then it induces a new Lie conformal algebraic structure on $\Ep$, given by 
\begin{align*}
\Omega^D_\lambda((a,u) ,(b,v))=(I-\Bar{D})\Omega_\lambda((I+\Bar{D})(a,u), (I+\Bar{D}(b,v)).     
\end{align*}

\begin{theorem}
Let $D: \Vs \to \Ws$ be a right deformation map on a quasi-twilled Lie conformal algebra $(\Ep,\Vs,\Ws)$. Then the product 
\begin{align}
 \{a_\la b\}^D_\Vs
  =&  \{a_\la b\}_\Vs + \rep_\Ws(D(a))_\la b - \rep_\Ws(D(b))_{-\la-\p}a  \label{eq:induced-bracket-I}   
\end{align}
defines  a Lie conformal algebra structure on $\Vs$. We denote it by $\Vs^D$. 
\end{theorem}

\begin{proof}
Let $a,b \in \Vs$. Then we have 
\begin{align*}
\Omega^D_\lambda((a,0) ,(b,0))=& (I-\Bar{D})\Omega_\lambda((I+\Bar{D})(a,0), (I+\Bar{D}(b,0))\\
=& (I-\Bar{D})\Omega_\lambda((a,D(a)),(b,D(b))) \\
=& (I-\Bar{D}) (
    \{a_\la b\}_\Vs + \rep_\Ws(D(a))_\la b
       - \rep_\Ws(D(b))_{-\la-\p}a,\\
    &\{D(a)_\la D(b)\}_\Ws + \rep_\Vs(a)_\la D(b)
       - \rep_\Vs(b)_{-\la-\p}D(a) + \psi_{1,\la}(a,b)) \\
=&( \{a_\la b\}_\Vs + \rep_\Ws(D(a))_\la b
       - \rep_\Ws(D(b))_{-\la-\p}a, 0) \\
=& (\{a_\la b\}^D_\Vs, 0).          
\end{align*}
Since $\Omega_\la$ is a Lie conformal bracket then $\{\cdot_\la \cdot\}^D_\Vs$ so is. Hence $(\Vs,\{\cdot_\la \cdot\}^D_\Vs)$ is a Lie conformal algebra. 
\end{proof}

\begin{remark}
The previous Theorem shows that $(\Ep,\Vs,\Ws,\Omega^D)$ is a quasi-twilled Lie conformal algebra, in which $\Vs$ is a subalgebra.      
\end{remark}

\begin{theorem}
Let $(\Ep,\Vs,\Ws)$ be a quasi-twilled Lie conformal algebra and $D: \Vs \to \Ws$ be a right deformation map. Define the map $\rep^D_\Vs: \Vs \to Cend(\Ws)$ by
\begin{align}
 \rep^D_\Vs(a)_\la u
  &= \rep_\Vs(a)_\la u + \{D(a)_\la u\}_\Ws + D(\rep_\Ws(u)_\la a). \label{eq:induced-rep-I}    
\end{align}
Then the pair $(\Ws, \rep^D_\Vs)$ is a representation of the Lie conformal algebra $\Vs^D$. 
\end{theorem}

\begin{proof}
We have already seen that  $(\Ep,\Vs,\Ws,\Omega^D)$ is a quasi-twilled Lie conformal algebra in which $\Vs$ is a  subalgebra  (and the induced Lie conformal algebra structure is $\Vs^D$).     Moreover, it follows that
the Lie conformal  algebra $\Vs^D$ has a representation on the $\mathbb{C}[\partial]$-module $\Vs$ with the action map given by
\begin{align*}
pr_{\Ws}(\Omega^D_\la((a,0),(0,u))=& pr_{\Ws}((I-\Bar{D})(\Omega_\la((a,D(a),(0,u))) \\
=& pr_{\Ws}((I-\Bar{D}) 
( - \rep_\Ws(u)_{-\lambda-\partial}a ,
          \{D(a)_\lambda u\}_\Ws + \rep_\Vs(a)_\lambda u )\\
          =&pr_{\Ws}(- \rep_\Ws(u)_{-\lambda-\partial}a, \{D(a)_\lambda u\}_\Ws + \rep_\Vs(a)_\lambda u + D(\rep_\Ws(u)_{-\lambda-\partial}a) \\
 =& \{D(a)_\lambda u\}_\Ws + \rep_\Vs(a)_\lambda u + D(\rep_\Ws(u)_{-\lambda-\partial}a).          
\end{align*}
This completes the proof. 
\end{proof}

Let $\delta^D:\Cch^k(\Vs,\Ws)\to\Cch^{k+1}(\Vs,\Ws)$ be the
Chevalley-Eilenberg coboundary of $(\Vs,\{\cdot_\la\cdot\}^D_\Vs)$
with coefficients in $(\Ws;\rep^D_\Vs)$.  For
$f\in\Cch^k(\Vs,\Ws)$ and $a_1,\ldots,a_{k+1}\in\Vs$:
%\begin{equation}\label{eq:CE-D-I}
%\begin{split}
  %(\delta^D f)_{\la_1,\ldots,\la_k}(a_1,\ldots,a_{k+1})
  %&= \sum_{i=1}^{k+1}(-1)^{i+1}
    % \rep^D_\Vs(a_i)_{\la_i}
    % f_{\widehat{\la_i}}(\ldots)\big|_{\la_{k+1}^\ddag}\\
  %&\quad + \sum_{i<j}(-1)^{i+j}
     %f_{\la_i+\la_j,\widehat{\la_i},\widehat{\la_j}}
     %(\{a_i{}_{\la_i}a_j\}^D_\Vs,\ldots)\big|_{\la_{k+1}^\ddag}.
%\end{split}
%\end{equation}
\begin{equation}\label{eq:CE-D-I}
\begin{split}
  (\delta^D f)_{\la_1,\ldots,\la_k}
  &(a_1,\ldots,a_{k+1}) \\
  &= \sum_{i=1}^{k}(-1)^{i+1}
     \rep^D_\Vs(a_i)_{\la_i}
     f_{\la_1,\ldots,\widehat{\la_i},\ldots \la_k}(a_1,\ldots,\widehat{a_i},\ldots,a_{k+1})
     %\big|_{\la_{k+1}\mapsto\la_{k+1}^\ddag} 
     \\
  &\quad
    + \sum_{i=1}^{k}(-1)^i
    f_{\la_1,\ldots,\widehat{\la_i},\ldots \la_k}
      (a_1,\ldots,\widehat{a_i},\ldots,a_k,
       \{a_{i\,\la_i} a_{k+1}\}^D_\Vs)
      %\big|_{\la_{k+1}\mapsto\la_{k+1}^\ddag} 
      \\
  &\quad
    + \!\!\sum_{1\leq i<j\leq k+1}\!\!(-1)^{k+i+j+1}
      f_{\la_1,\ldots,\widehat{\la_i},\ldots ,\widehat{\la_j},\ldots, \la_k,\la_{k+1}^\ddag}
      (a_1,\ldots,a_k,\{a_{i\,\la_i}a_j\}^D_\Vs)
      %\big|_{\la_{k+1}\mapsto\la_{k+1}^\ddag} 
      \\
  &\quad
    + (-1)^k\rep^D_\Vs(a_{k+1})_{\la_{k+1}^\ddag}
      f_{\la_1,\ldots,\la_{k-1}}(a_1,\ldots,a_k),
\end{split}
\end{equation}
Expanding via \eqref{eq:induced-bracket-I}-\eqref{eq:induced-rep-I}
gives, for all $a_1,\ldots,a_{k+1}\in\Vs$:
\begin{equation*}
\begin{split}
  &(\delta^D f)_{\la_1,\ldots,\la_k}
  (a_1,\ldots,a_{k+1}) \\
  =& \sum_{i=1}^{k}(-1)^{i+1}
     \rep_\Vs(a_i)_{\la_i}
     f_{\la_1,\ldots,\widehat{\la_i},\ldots \la_k}(a_1,\ldots,\widehat{a_i},\ldots,a_{k+1})
     %\big|_{\la_{k+1}\mapsto\la_{k+1}^\ddag} 
     \\
  &+\sum_{i=1}^{k}(-1)^{i+1}
     \{D(a_i)_{\la_i}
     f_{\la_1,\ldots,\widehat{\la_i},\ldots \la_k}(a_1,\ldots,\widehat{a_i},\ldots,a_{k+1})\}_\Ws
     %\big|_{\la_{k+1}\mapsto\la_{k+1}^\ddag} 
     \\
  &+\sum_{i=1}^{k}(-1)^{i+1}
D\rep_\Ws(f_{\la_1,\ldots,\widehat{\la_i},\ldots \la_k}(a_1,\ldots,\widehat{a_i},\ldots,a_{k+1}))_{\la_i}(a_i)     
     %\big|_{\la_{k+1}\mapsto\la_{k+1}^\ddag} 
     \\
  &
    + \sum_{i=1}^{k}(-1)^i
    f_{\la_1,\ldots,\widehat{\la_i},\ldots \la_k}
      (a_1,\ldots,\widehat{a_i},\ldots,a_k,
       \{a_{i\,\la_i} a_{k+1}\}_\Vs)
      %\big|_{\la_{k+1}\mapsto\la_{k+1}^\ddag} 
      \\&
    + \sum_{i=1}^{k}(-1)^i
    f_{\la_1,\ldots,\widehat{\la_i},\ldots \la_k}
      (a_1,\ldots,\widehat{a_i},\ldots,a_k,
       \rep_\Ws(D(a_{i})_{\la_i} a_{k+1}-\rep_\Ws(D(a_{k+1})_{-\la_i-\partial} a_{i})
      %\big|_{\la_{k+1}\mapsto\la_{k+1}^\ddag} 
      \\
  &
    + \!\!\sum_{1\leq i<j\leq k+1}\!\!(-1)^{k+i+j+1}
      f_{\la_1,\ldots,\widehat{\la_i},\ldots ,\widehat{\la_j},\ldots, \la_k,\la_{k+1}^\ddag}
      (a_1,\ldots,a_k,\{a_{i\,\la_i}a_j\}_\Vs)
      %\big|_{\la_{k+1}\mapsto\la_{k+1}^\ddag} 
      \\ &
    + \!\!\sum_{1\leq i<j\leq k+1}\!\!(-1)^{k+i+j+1}
      f_{\la_1,\ldots,\widehat{\la_i},\ldots ,\widehat{\la_j},\ldots, \la_k,\la_{k+1}^\ddag}
      (a_1,\ldots,a_k,\rep_\Ws(D(a_{i})_{\la_i} a_{j}-\rep_\Ws(D(a_{j})_{-\la_i-\partial} a_{i})
      %\big|_{\la_{k+1}\mapsto\la_{k+1}^\ddag} 
      \\
      &+(-1)^{k}
     \rep_\Vs(a_{k+1})_{\la_{k+1}}
     f_{\la_1,\ldots \la_k}(a_1,\ldots,a_{k+1})
     %\big|_{\la_{k+1}\mapsto\la_{k+1}^\ddag} 
     +(-1)^{k}
     \{D(a_{k+1})_{\la_{k+1}}
     f_{\la_1,\ldots \la_k}(a_1,\ldots,a_{k})\}_\Ws
     %\big|_{\la_{k+1}\mapsto\la_{k+1}^\ddag} 
     \\
  &+(-1)^{k}
D\rep_\Ws(f_{\la_1,\ldots \la_k}(a_1,\ldots,\widehat{a_i},\ldots,a_{k}))_{\la_{k+1}}(a_i),
\end{split}
\end{equation*}

\begin{definition}\label{def:cohom-I}
Let $D:\Vs\to\Ws$ be a right deformation map\ of $(\Ep,\Vs,\Ws)$.  Set $C^0(D):=0$,
$C^1(D):=\Ws$, and $C^n(D):=\Cch^{n-1}(\Vs,\Ws)$ for $n\geq2$.
The \textit{cohomology of the right deformation map\ $D$} is the cohomology of the
cochain complex $\bigl(\bigoplus_{i\geq0}C^i(D),\,\delta^D\bigr)$, with
cohomology groups $H^n(D)$ for $n\geq0$.
\end{definition}
 
An element $u\in\Ws=C^1(D)$ is a $1$-cocycle if and only if
\[
  \{D(a)_\la u\}_\Ws + \rep_\Vs(a)_\la u + D(\rep_\Ws(u)_\la a)
  = 0, \quad \forall\,a\in\Vs,
\]
and $f\in\Cch^1(\Vs,\Ws)=C^2(D)$ is a $2$-cocycle if and only if
\begin{align*}
  &\{D(a)_\la f(b)\}_\Ws - \{D(b)_{-\la-\p}f(a)\}_\Ws
  + \rep_\Vs(a)_\la f(b) - \rep_\Vs(b)_{-\la-\p}f(a)\\
  &\quad - D(\rep_\Ws(f(a))_\la b) + D(\rep_\Ws(f(b))_{-\la-\p}a)
  = f(\{a_\la b\}_\Vs)
   + f(\rep_\Ws(D(a))_\la b - \rep_\Ws(D(b))_{-\la-\p}a),
\end{align*}
for all $a,b\in\Vs$.
 
%\begin{proposition}\label{prop:l1D-CE-I}
%For any $f\in\Cch^k(\Vs,\Ws)$, one has
%$l_1^D(f) = (-1)^{k-1}\delta^D f.$
%\end{proposition}
 
%\begin{proof}
%By \eqref{eq:l1D-I} and Theorem~\ref{thm:control-I},
%$l_1^D(f) = [\smu,\widehat{f}]_{NR} + [[\smv,\widehat{D}]_{NR},\widehat{f}]_{NR}
%= [\smu^D,\widehat{f}]_{NR} = (-1)^{k-1}\delta^D f$.
%\end{proof}
 
We now illustrate that Definition~\ref{def:cohom-I} unifies the
cohomologies of the classical conformal operators.
 
\begin{example}[Cohomology of a  modified $r$-matrix]\label{ex:Coh modified-r}
In the setting of Example~\ref{ex:modified-r}, let $D:\Vs\to\Vs$ be
a modified $r$-matrix.  Then $(\Vs,\{\cdot_\la\cdot\}^D_\Vs)$
is a Lie conformal  algebra with
$\{a_\la b\}^D_\Vs = \{a_\la D(b)\}_\Vs - \{b_{-\la-\p}D(a)\}_\Vs$,
and the module action is
$\rep^D_\Vs(a)_\la b = \{D(a)_\la b\}_\Vs - D(\{b_\la a\}_\Vs)$.
The cohomology of Definition~\ref{def:cohom-I} applied to these structures  is the
\textit{cohomology for a  modified $r$-matrix} \cite{JS2}. The coboundary operators are defined as 

\begin{equation}\label{Coh: modified r-matrix}
\begin{split}
  &(\delta^D f)_{\la_1,\ldots,\la_k}
  (a_1,\ldots,a_{k+1}) \\
  &=\sum_{i=1}^{k}(-1)^{i+1}
     \{D(a_i)_{\la_i}
     f_{\la_1,\ldots,\widehat{\la_i},\ldots \la_k}(a_1,\ldots,\widehat{a_i},\ldots,a_{k+1})\}_\Vs
     %\big|_{\la_{k+1}\mapsto\la_{k+1}^\ddag} 
     \\
  &+\sum_{i=1}^{k}(-1)^{i+1}
D(\{a_{i\,\la_i}f_{\la_1,\ldots,\widehat{\la_i},\ldots \la_k}(a_1,\ldots,\widehat{a_i},\ldots,a_{k+1})\}_\Vs)     
     %\big|_{\la_{k+1}\mapsto\la_{k+1}^\ddag} 
     \\
  &
    + \sum_{i=1}^{k}(-1)^i
    f_{\la_1,\ldots,\widehat{\la_i},\ldots \la_k}
      (a_1,\ldots,\widehat{a_i},\ldots,a_k,
       \{D(a_{i})_{\la_i} a_{k+1}\}_\Vs-\{D(a_{k+1})_{-\la_i-\partial} a_{i}\}_\Vs)
      %\big|_{\la_{k+1}\mapsto\la_{k+1}^\ddag} 
      \\
  &
    + \!\!\sum_{1\leq i<j\leq k+1}\!\!(-1)^{k+i+j+1}
      f_{\la_1,\ldots,\widehat{\la_i},\ldots ,\widehat{\la_j},\ldots, \la_k,\la_{k+1}^\ddag}
      (a_1,\ldots,a_k,\{D(a_{i})_{\la_i} a_{j}\}_\Vs-\{D(a_{j})_{-\la_i-\partial} a_{i}\}_\Vs)
      %\big|_{\la_{k+1}\mapsto\la_{k+1}^\ddag} 
      \\
      &
     +(-1)^{k}\big(
     \{D(a_{k+1})_{\la_{k+1}}
     f_{\la_1,\ldots \la_k}(a_1,\ldots,a_{k})\}_\Vs
     %\big|_{\la_{k+1}\mapsto\la_{k+1}^\ddag} 
     +
D(\{f_{\la_1,\ldots \la_k}(a_1,\ldots,\widehat{a_i},\ldots,a_{k})_{\la_{k+1}}a_i\}_\Vs)\big).
\end{split}
\end{equation}

\end{example}
 
\begin{example}[Cohomology of a  crossed homomorphism]\label{ex:coh crossed-hom}
In the setting of Example~\ref{ex:crossed-hom}, let $D:\Vs\to\Ws$ be
a conformal crossed homomorphism.  The Lie conformal 
algebra $(\Vs,\{\cdot_\la\cdot\}_\Vs)$ acts on $\Ws$ via
$\rep^D_\Vs(a)_\la u = \rep(a)_\la u + \{D(a)_\la u\}_\Ws$.
 The coboundary operators are defined as 
\begin{equation*}\label{Coh: crossed homomorphism}
\begin{split}
  &(\delta^D f)_{\la_1,\ldots,\la_k}
  (a_1,\ldots,a_{k+1}) \\
  =& \sum_{i=1}^{k}(-1)^{i+1}
     \rep_\Vs(a_i)_{\la_i}
     f_{\la_1,\ldots,\widehat{\la_i},\ldots \la_k}(a_1,\ldots,\widehat{a_i},\ldots,a_{k+1})
     %\big|_{\la_{k+1}\mapsto\la_{k+1}^\ddag} 
     \\
  &+\sum_{i=1}^{k}(-1)^{i+1}
D\rep_\Ws(f_{\la_1,\ldots,\widehat{\la_i},\ldots \la_k}(a_1,\ldots,\widehat{a_i},\ldots,a_{k+1}))_{\la_i}(a_i)     
     %\big|_{\la_{k+1}\mapsto\la_{k+1}^\ddag} 
     \\
  &
    + \sum_{i=1}^{k}(-1)^i
    f_{\la_1,\ldots,\widehat{\la_i},\ldots \la_k}
      (a_1,\ldots,\widehat{a_i},\ldots,a_k,
       \{a_{i\,\la_i} a_{k+1}\}_\Vs)
      %\big|_{\la_{k+1}\mapsto\la_{k+1}^\ddag} 
      \\ 
      &+(-1)^{k}
     \rep_\Vs(a_{k+1})_{\la_{k+1}}
     f_{\la_1,\ldots \la_k}(a_1,\ldots,a_{k+1})
     %\big|_{\la_{k+1}\mapsto\la_{k+1}^\ddag} 
     +(-1)^{k}
     \{D(a_{k+1})_{\la_{k+1}}
     f_{\la_1,\ldots \la_k}(a_1,\ldots,a_{k})\}_\Ws
 .
\end{split}
\end{equation*}
Notice that this case was considered in \cite{BBSS}.
\end{example}
 
\begin{example}[Cohomology of a  derivation]\label{ex:coh conf-derivation}
In the setting of Example~\ref{ex:conf-derivation}, the cohomology associated with a right deformation map
$D:\Vs\to\Mod$ is given by the Chevalley--Eilenberg cohomology of
$(\Vs,\{\cdot,\cdot\}_\Vs)$ with coefficients in the representation
$(\Mod;\rep)$. In the classical case, the corresponding coboundary
operators are explicitly described in \cite{TFS}. 

%\begin{equation}\label{Coh: derivation}
%\begin{split}
  %&(\delta^D f)_{\la_1,\ldots,\la_k}
  %(a_1,\ldots,a_{k+1}) \\
 % =& \sum_{i=1}^{k}(-1)^{i+1}
   %  \rep_\Vs(a_i)_{\la_i}
    % f_{\la_1,\ldots,\widehat{\la_i},\ldots \la_k}(a_1,\ldots,\widehat{a_i},\ldots,a_{k+1}) 
     %\\
 % &+\sum_{i=1}^{k}(-1)^{i+1}
     %\{D(a_i)_{\la_i}
     %f_{\la_1,\ldots,\widehat{\la_i},\ldots \la_k}%(a_1,\ldots,\widehat{a_i},\ldots,a_{k+1})\}_\Ws
     %\\
  %&+\sum_{i=1}^{k}(-1)^{i+1}
%D\rep_\Ws(f_{\la_1,\ldots,\widehat{\la_i},\ldots \la_k}(a_1,\ldots,\widehat{a_i},\ldots,a_{k+1}))_{\la_i}(a_i)      
     %\\
  %&
    %+ \sum_{i=1}^{k}(-1)^i
    %f_{\la_1,\ldots,\widehat{\la_i},\ldots \la_k}
     % (a_1,\ldots,\widehat{a_i},\ldots,a_k,
       %\{a_{i\,\la_i} a_{k+1}\}_\Vs) 
      %\\&
    %+ \sum_{i=1}^{k}(-1)^i
    %f_{\la_1,\ldots,\widehat{\la_i},\ldots \la_k}
      %(a_1,\ldots,\widehat{a_i},\ldots,a_k,
       %\rep_\Ws(D(a_{i})_{\la_i} a_{k+1}-\rep_\Ws(D(a_{k+1})_{-\la_i-\partial} a_{i})
      %\\
  %&
   % + \!\!\sum_{1\leq i<j\leq k+1}\!\!(-1)^{k+i+j+1}
      %f_{\la_1,\ldots,\widehat{\la_i},\ldots ,\widehat{\la_j},\ldots, \la_k,\la_{k+1}^\ddag}
      %(a_1,\ldots,a_k,\{a_{i\,\la_i}a_j\}_\Vs) 
      %\\ &
    %+ \!\!\sum_{1\leq i<j\leq k+1}\!\!(-1)^{k+i+j+1}
      %f_{\la_1,\ldots,\widehat{\la_i},\ldots ,\widehat{\la_j},\ldots, \la_k,\la_{k+1}^\ddag}
      %(a_1,\ldots,a_k,\rep_\Ws(D(a_{i})_{\la_i} a_{j}-\rep_\Ws(D(a_{j})_{-\la_i-\partial} a_{i}) 
      %\\
      %&+(-1)^{k}
     %\rep_\Vs(a_{k+1})_{\la_{k+1}}
     %f_{\la_1,\ldots \la_k}(a_1,\ldots,a_{k+1})
     %+(-1)^{k}
     %\{D(a_{k+1})_{\la_{k+1}}
     %f_{\la_1,\ldots \la_k}(a_1,\ldots,a_{k})\}_\Ws 
     %\\
  %&+(-1)^{k}
%D\rep_\Ws(f_{\la_1,\ldots \la_k}(a_1,\ldots,\widehat{a_i},\ldots,a_{k}))_{\la_{k+1}}(a_i),
%\end{split}
%\end{equation}

\end{example}
 
\begin{example}[Cohomology of a Lie conformal  algebra homomorphism]\label{ex:coh conf-hom}
In the setting of Example~\ref{ex:conf-hom}, $\{a_\la b\}^D_\Vs
= \{a_\la b\}_\Vs$ and $\rep^D_\Vs(a)_\la u = \{D(a)_\la u\}_\Ws$.
The cohomology given in Definition~\ref{def:cohom-I} with these special structures  is the
\textit{cohomology for a Lie conformal  algebra homomorphism}.  The coboundary operators, given explicitly, are defined as 

\begin{equation}\label{Coh: homomorphism}
\begin{split}
  &(\delta^D f)_{\la_1,\ldots,\la_k}
  (a_1,\ldots,a_{k+1}) \\
  =& \sum_{i=1}^{k}(-1)^{i+1}
     \{D(a_i)_{\la_i}
     f_{\la_1,\ldots,\widehat{\la_i},\ldots \la_k}(a_1,\ldots,\widehat{a_i},\ldots,a_{k+1})\}_\Ws
     %\big|_{\la_{k+1}\mapsto\la_{k+1}^\ddag} 
     \\
  &
    + \sum_{i=1}^{k}(-1)^i
    f_{\la_1,\ldots,\widehat{\la_i},\ldots \la_k}
      (a_1,\ldots,\widehat{a_i},\ldots,a_k,
       \{a_{i\,\la_i} a_{k+1}\}_\Vs) 
      \\
  &
    + \!\!\sum_{1\leq i<j\leq k+1}\!\!(-1)^{k+i+j+1}
      f_{\la_1,\ldots,\widehat{\la_i},\ldots ,\widehat{\la_j},\ldots, \la_k,\la_{k+1}^\ddag}
      (a_1,\ldots,a_k,\{a_{i\,\la_i}a_j\}_\Vs)
      \\
      &
     +(-1)^{k}
     \{D(a_{k+1})_{\la_{k+1}}
     f_{\la_1,\ldots \la_k}(a_1,\ldots,a_{k})\}_\Ws
     .
\end{split}
\end{equation}
So far, this cohomology is new.
\end{example}

%--------------------------------------------------------------------------
\subsection{The controlling curved $L_\infty$-algebra of right deformation maps}
\label{subsec:control-I}
%--------------------------------------------------------------------------
 
We recall Voronov's derived bracket construction \cite{Vo}.
 
\begin{definition}[Curved $V$-data, {\cite{Vo}}]\label{def:cVdata-I}
A \textit{curved $V$-data} is a quadruple $(L,F,\mathbf{P},\Delta)$,
where $(L,[\cdot,\cdot])$ is a graded Lie algebra, $F$ is an abelian
graded Lie sub-algebra, $\mathbf{P}:L\to L$ is a projection with
$\mathrm{Im}(\mathbf{P})=F$ and $\ker(\mathbf{P})$ is a graded Lie
sub-algebra, and $\Delta\in L^1$ with $[\Delta,\Delta]=0$.  When
moreover $\Delta\in\ker(\mathbf{P})^1$, it is a \textit{$V$-data}.
\end{definition}
 
\begin{theorem}[{\cite{Vo}}]\label{thm:Voronov}
Let $(L,F,\mathbf{P},\Delta)$ be a curved $V$-data.  Then
$(F,\{l_k\}_{k\geq0})$ is a curved $L_\infty$-algebra with
\[
  l_0 = \mathbf{P}(\Delta), \qquad
  l_k(f_1,\ldots,f_k)
  = \mathbf{P}([\cdots[[\Delta,f_1]_{NR},f_2]_{NR},\cdots,f_k]_{NR}),
  \quad k\geq1.
\]
\end{theorem}
 
\begin{theorem}\label{thm:control-I}
Let $(\Ep,\Vs,\Ws)$ be a quasi-twilled Lie conformal  algebra with
$\str=\shf+\smu+\smv$.  Set
\[
  L := \bigoplus_{n\geq0}\Cch^{n+1}(\Ep,\Ep), \qquad
  F := \bigoplus_{n\geq0}\Cch^{n+1}(\Vs,\Ws),
\]
let $\mathbf{P}:L\to L$ be the projection onto $F$, and set
$\Delta:=\str$.  Then $(L,F,\mathbf{P},\Delta)$ is a curved $V$-data
and the resulting curved $L_\infty$-algebra on $F$ is
\[
  \Bigl(\bigoplus_{n\geq0}\Cch^{n+1}(\Vs,\Ws),\; l_0,\, l_1,\, l_2\Bigr),
\]
with
\begin{align}
  l_0 &= \shf, \label{eq:l0-I}\\
  l_1(f) &= [\smu,\,\widehat{f}\,]_{NR}, \label{eq:l1-I}\\
  l_2(f,g) &= [[\smv,\,\widehat{f}\,]_{NR},\,\widehat{g}\,]_{NR},
  \label{eq:l2-I}
\end{align}
for all $f\in\Cch^m(\Vs,\Ws)$, $g\in\Cch^n(\Vs,\Ws)$, and $l_k=0$ for
$k\geq3$.
 
A $\Rg$-module homomorphism $D:\Vs\to\Ws$ is a right deformation map\ of
$(\Ep,\Vs,\Ws)$ if and only if $D$ is a Maurer-Cartan element of this
curved $L_\infty$-algebra:
\begin{equation}\label{eq:MC-I}
  l_0 + l_1(D) + \tfrac{1}{2}l_2(D,D) = 0.
\end{equation}
\end{theorem}
 
\begin{proof}
It is clear that $F=\bigoplus_{n\geq0}\Cch^{n+1}(\Vs,\Ws)$ is an
abelian graded Lie sub-algebra of $(L,[\cdot,\cdot]_{NR})$: any two
elements of $F$ have bidegrees $(-1\mid\bullet)$, so by
Lemma~\ref{lem:FN-zero} (Preliminaries) their NR bracket vanishes.
Since $\mathbf{P}^2=\mathbf{P}$ and $[\str,\str]_{NR}=0$, the data
$(L,F,\mathbf{P},\str)$ is a curved $V$-data.  By
Theorem~\ref{thm:Voronov},
\[
  l_0 = \mathbf{P}(\str) = \shf, \qquad
  l_1(f) = \mathbf{P}([\str,\widehat{f}]_{NR}) = [\smu,\widehat{f}]_{NR},
\]
since only the $\smu$-part of $\str$ contributes to the projection
(the $\smv$ and $\shf$ parts land outside $F$ after one bracket
with an element of $F$).  Similarly,
$l_2(f,g)=[[\smv,\widehat{f}]_{NR},\widehat{g}]_{NR}$, and $l_k=0$
for $k\geq3$ by abelianess of $F$ and Lemma~\ref{lem:FN-bigrade}.
 
Expanding \eqref{eq:MC-I} via \eqref{eq:l0-I}-\eqref{eq:l2-I} and
using the explicit NR-bracket formulas recovers exactly \eqref{eq:DmapI},
completing the proof.
\end{proof}
 
\begin{corollary}[Controlling algebra for  modified $r$-matrices]
\label{cor:control-modified-r}
In the setting of Example~\ref{ex:modified-r},
$\bigl(\bigoplus_{n\geq0}\Cch^{n+1}(\Vs,\Vs),\,l_0,\,l_1{=}0,\,l_2\bigr)$
is a curved $L_\infty$-algebra where $l_0=\alpha\{\cdot_\la\cdot\}_\Vs$
and $l_2$ is induced by $[\cdot,\cdot]_\Vs$ via the 
NR
%FN
mechanism.
Maurer-Cartan elements are exactly  modified $r$-matrices.
\end{corollary}
 
\begin{corollary}[Controlling algebra for  crossed homomorphisms]
\label{cor:control-crossed}
In the setting of Example~\ref{ex:crossed-hom},
$\bigl(\bigoplus_{n\geq1}\Cch^n(\Vs,\Ws),\,\dM,\,\{\cdot,\cdot\}\bigr)$
is a differential graded Lie algebra, where the differential
$\dM:\Cch^p(\Vs,\Ws)\to\Cch^{p+1}(\Vs,\Ws)$ is
\begin{equation}\label{eq:diff-I}
\begin{split}
  (\dM f)_{\la_1,\ldots,\la_p}(a_1,\ldots,a_{p+1})
  &= \sum_{i=1}^{p+1}(-1)^{p+i}\rep(a_i)_{\la_i}
     f_{\widehat{\la_i}}(a_1,\ldots,\widehat{a_i},\ldots,a_{p+1})
     \big|_{\la_{p+1}^\ddag}\\
  &\quad + \sum_{i<j}(-1)^{p+i+j-1}
     f_{\la_i+\la_j,\widehat{\la_i},\widehat{\la_j}}
     (\{a_i{}_{\la_i}a_j\}_\Vs,a_1,\ldots,
     \widehat{a_i},\ldots,\widehat{a_j},\ldots,a_{p+1})
     \big|_{\la_{p+1}^\ddag},
\end{split}
\end{equation}
and the graded bracket is induced by $\{\cdot_\la\cdot\}_\Ws$.
This is the \textit{controlling algebra for  crossed
homomorphisms of weight $\alpha$}.
\end{corollary}
 
\begin{corollary}[Controlling algebra for  derivations]
\label{cor:control-deriv}
In the setting of Example~\ref{ex:conf-derivation}, the controlling
structure is the cochain complex
$\bigl(\bigoplus_{n\geq1}\Cch^n(\Vs,\Mod),\,\dM\bigr)$
with $\dM$ the Chevalley-Eilenberg coboundary of $(\Vs,\{\cdot_\la\cdot\}_\Vs)$
with values in $(\Mod;\rep)$.  This is the \textit{controlling algebra
for derivations} from $(\Vs,\{\cdot_\la\cdot\}_\Vs)$ to
$(\Mod;\rep)$.
\end{corollary}
 
\begin{corollary}
\label{cor:control-hom}
In the setting of Example~\ref{ex:conf-hom}, the controlling structure
is a differential graded Lie algebra
$\bigl(\bigoplus_{n\geq1}\Cch^n(\Vs,\Ws),\,\dM,\,\{\cdot,\cdot\}\bigr)$,
where $\dM$ is the Chevalley-Eilenberg coboundary for $\Vs$ acting
trivially on $\Ws$.  This is the \textit{controlling algebra for
Lie conformal  algebra homomorphisms}.
\end{corollary}
 
Let $D:\Vs\to\Ws$ be a right deformation map.  By Theorem~\ref{thm:control-I}, $D$ is
a Maurer-Cartan element of the curved $L_\infty$-algebra above.  The
twisted $L_\infty$-algebra \cite{DSV,Get} has structure maps
\begin{align}
  l_1^D(f) &= l_1(f) + l_2(D,f), \label{eq:l1D-I}\\
  l_2^D(f,g) &= l_2(f,g), \label{eq:l2D-I}\\
  l_k^D &= 0, \quad k\geq3. \label{eq:lkD-I}
\end{align}
 
\begin{theorem}\label{thm:deform-DmapI}
Let $D:\Vs\to\Ws$ be a right deformation map\ of $(\Ep,\Vs,\Ws)$.  A $\Rg$-module
homomorphism $D':\Vs\to\Ws$ gives a right deformation map\ $D+D'$ if and only if $D'$
is a Maurer-Cartan element of the twisted $L_\infty$-algebra
$\bigl(\bigoplus_{n\geq0}\Cch^{n+1}(\Vs,\Ws),\,l_1^D,\,l_2^D\bigr)$:
\begin{equation}\label{eq:MC-deform-I}
  l_1^D(D') + \tfrac{1}{2}l_2^D(D',D') = 0.
\end{equation}
\end{theorem}
 
\begin{proof}
By Theorem~\ref{thm:control-I}, $D+D'$ is a right deformation map\ if and only if
$l_0+l_1(D+D')+\frac{1}{2}l_2(D+D',D+D')=0$.  Since $D$ satisfies
\eqref{eq:MC-I}, this reduces to
$l_1(D')+l_2(D,D')+\frac{1}{2}l_2(D',D')=0$, i.e.\ \eqref{eq:MC-deform-I}.
\end{proof}
 
\begin{remark}
Applying Theorem~\ref{thm:deform-DmapI} to
Corollaries~\ref{cor:control-modified-r}-\ref{cor:control-hom}
recovers the differential graded Lie algebras governing deformations of
conformal modified $r$-matrices (new), conformal crossed homomorphisms,
conformal derivations, and Lie conformal  algebra homomorphisms.
\end{remark}

%==========================================================================
\section{Controlling algebras and cohomologies of left deformation maps}
\label{sec:typeII}
%==========================================================================
 
In this section, we introduce the notion of left deformation map on a quasi-twilled Lie conformal algebra, unifying relative Rota-Baxter operators, twisted Rota-Baxter operators, Reynolds operators, and deformation maps of matched pairs of Lie conformal algebras. We also describe the controlling algebra of left deformation maps on a quasi-twilled Lie conformal algebra, showing that it carries the structure of a curved $L_\infty$-algebra. Throughout this section, $(\Ep,\Vs,\Ws)$ denotes a quasi-twilled Lie conformal algebra with structure $\str=\shf+\smu+\smv$. 
 
%--------------------------------------------------------------------------
\subsection{Left deformation maps}
\label{subsec:DmapII}
%--------------------------------------------------------------------------
 
Let $B:\Ws\to\Vs$ be a $\Rg$-module homomorphism.  Its \textit{graph} is the $\Rg$-submodule
\[
  \mathrm{Gr}(B) := \{(B(u),u)\mid u\in\Ws\} \subset \Vs\oplus\Ws = \Ep.
\]
 
\begin{definition}\label{def:DmapII}
Let $(\Ep,\Vs,\Ws)$ be a quasi-twilled  Lie conformal algebra.
A left deformation map  of $(\Ep,\Vs,\Ws)$
is a $\Rg$-module homomorphism $B:\Ws\longrightarrow\Vs$
such that its graph is a Lie conformal subalgebra of $\Ep$. 
\end{definition}

%===
\begin{proposition}\label{left Dmap identity}
A $\Rg$-module homomorphism $B:\Ws\to\Vs$ is a left deformation map  if and only if, for all $u,v\in\Ws$, we have
\begin{align}
  & \{B(u)_\la B(v)\}_\Vs
  + \rep_\Ws(u)_\la B(v) - \rep_\Ws(v)_{-\la-\p}B(u) \nonumber\\
  =&  B(\{u_\la v\}_\Ws + \rep_\Vs(B(u))_\la v
             - \rep_\Vs(B(v))_{-\la-\p}u
             + \psi_{1,\la}(B(u),B(v))), \label{eq:DmapII}  
\end{align}
\end{proposition}

\begin{proof}
For any $(B(u),u), (B(v),v) \in \mathrm{Gr}(B)$, we have
\begin{align*}
\Omega_\la((B(u),u), (B(v),v))=&(\{B(u)_\la B(v)\}_\Vs
  + \rep_\Ws(u)_\la B(v) - \rep_\Ws(v)_{-\la-\p}B(u),\\
  & \{u_\la v\}_\Ws + \rep_\Vs(B(u))_\la v
             - \rep_\Vs(B(v))_{-\la-\p}u
             + \psi_{1,\la}(B(u),B(v))).
\end{align*}
Then $B: \Ws\to\Vs$ is a left deformation map if and only if the $\Vs$-component equals
$B$ applied to the $\Ws$-component, which is precisely Eq. \eqref{eq:DmapII}.
\end{proof}
 
\begin{remark}
  The two types are related by inversion: if
$D:\Vs\to\Ws$ is an invertible right deformation map, then $D^{-1}:\Ws\to\Vs$ is a
left deformation map of the same quasi-twilled structure.
\end{remark}
 
\begin{example}[Relative Rota-Baxter operator]\label{ex:conf-RRB}
Let $\Ep=\Vs\ltimes_\rep\Ws$ with $\{u_\la v\}_\Ws=\la\{u_\la v\}_{\Ws,0}$,
$\rep_\Ws=0$, $\psi_1=0$.
A $\Rg$-module homomorphism $B:\Ws\to\Vs$ is a left deformation map if and only if
\begin{equation}\label{eq:conf-RRB}
  \{B(u)_\la B(v)\}_\Vs
  = B\!\bigl(\rep(B(u))_\la v - \rep(B(v))_{-\la-\p}u
             + \la\{u_\la v\}_\Ws\bigr),
  \quad \forall\,u,v\in\Ws,
\end{equation}
i.e.\ $B$ is a \textit{conformal relative Rota-Baxter operator of weight $\la$}
on $(\Vs,\{\cdot_\la\cdot\}_\Vs)$ with respect to the module $(\Ws;\rep)$.
For $\la=0$ this gives the conformal $\mathcal{O}$-operator.
\end{example}
 
\begin{example}[Conformal twisted Rota-Baxter operator]\label{ex:conf-TRB}
Let $\Ep=\Vs\ltimes_\phi\Mod$ with $\rep_\Ws=0$, $\{u_\la v\}_\Ws=0$,
$\psi_1=\phi\in\Cch^2(\Vs,\Mod)$ a $2$-cocycle.
A $\Rg$-module homomorphism $B:\Mod\to\Vs$ is a left deformation map if and only if
\begin{equation}\label{eq:conf-TRB}
  \{B(m)_\la B(n)\}_\Vs
  = B\!\bigl(\rep(B(m))_\la n - \rep(B(n))_{-\la-\p}m
             + \phi_\la(B(m),B(n))\bigr),
  \quad \forall\,m,n\in\Mod,
\end{equation}
i.e.\ $B$ is a \textit{conformal twisted Rota-Baxter operator}
associated to the cocycle $\phi$.
\end{example}
 
\begin{example}[Conformal Reynolds operator]\label{ex:conf-Reynolds}
Let $\Ep=\Vs\ltimes_{\ad,\{\cdot_\la\cdot\}_\Vs}\Vs$ with $\rep=\ad$,
$\{u_\la v\}_\Ws=\{u_\la v\}_\Vs$, $\psi_{1,\la}(a,b)=\{a_\la b\}_\Vs$.
A $\Rg$-module homomorphism $B:\Vs\to\Vs$ is a left deformation map\ if and only if
\begin{equation}\label{eq:conf-Reynolds}
  \{B(u)_\la B(v)\}_\Vs
  = B\!\bigl(\{B(u)_\la v\}_\Vs - \{B(v)_{-\la-\p}u\}_\Vs
             + \{B(u)_\la B(v)\}_\Vs\bigr),
  \quad \forall\,u,v\in\Vs,
\end{equation}
i.e.\ $B$ is a \textit{conformal Reynolds operator}.
\end{example}
 
\begin{example}[Deformation map of a matched pair of Lie conformal  algebras]
\label{ex:conf-matched-pair}
Let $\Ep=\Vs\bowtie\Ws$ be a twilled Lie conformal  algebra
(so $\psi_1=0$).  A $\Rg$-module homomorphism $B:\Ws\to\Vs$ is a
left deformation map\ if and only if
\begin{equation}\label{eq:conf-matched-map}
  \{B(u)_\la B(v)\}_\Vs
  + \rep_\Ws(u)_\la B(v) - \rep_\Ws(v)_{-\la-\p}B(u)
  = B\!\bigl(\{u_\la v\}_\Ws + \rep_\Vs(B(u))_\la v
             - \rep_\Vs(B(v))_{-\la-\p}u\bigr),
\end{equation}
i.e.\ $B$ is a \textit{deformation map of a matched pair of conformal
Lie algebras}.
\end{example}

%--------------------------------------------------------------------------
\subsection{Cohomology of left deformation maps}
\label{subsec:cohom-I}
%--------------------------------------------------------------------------

%======Induced structures===========
Let $(\Ep,\Vs,\Ws)$ be a quasi-twilled Lie conformal algebra and
$B: \Ws\to\Vs$ be a left deformation map. Define the map $\Bar{B}: \Ep \to \Ep,\ (a,u) \mapsto (B(u),0)$. It is  clear that $\Bar{B}^2=0$. 
Consider the map $I+\Bar{B}: \Ep \to \Ep,\ (a,u) \mapsto (a+B(u),u)$. It is an invertible $\mathbb{C}[\partial]$-homomorphism with inverse $I-\Bar{B}$. Then it induces a new Lie conformal algebraic structure on $\Ep$, given by 
\begin{align*}
\Omega^B_\lambda((a,u) ,(b,v))=(I-\Bar{B})\Omega_\lambda((I+\Bar{B})(a,u), (I+\Bar{B}(b,v)).     
\end{align*}

\begin{theorem}
Let $B: \Ws \to \Vs$ be left deformation map on a quasi-twilled Lie conformal algebra $(\Ep,\Vs,\Ws)$. Then the $\la$-product 
\begin{align}
 \{u_\la v\}^B_\Ws
  &= \{u_\la v\}_\Ws + \rep_\Vs(B(u))_\la v - \rep_\Vs(B(v))_{-\la-\p}u
     + \psi_{1,\la}(B(u),B(v)), 
  \label{eq:induced-bracket-II}   
\end{align}
is a Lie conformal algebra structure on $\Ws$. We denote it by $\Ws^B$. 
\end{theorem}

\begin{proof}
Let $u,v \in \Ws$. Then we have
\begin{align*}
\Omega^B_\lambda((0,u) ,(0,v))=& (I-\Bar{B})\Omega_\lambda((I+\Bar{B})(0,u), (I+\Bar{B}(0,u))\\
=& (I-\Bar{B})\Omega_\lambda((B(u),u),(B(v),v)) \\
=&(I-\Bar{B})(\{B(u)_\la B(v)\}_\Vs
  + \rep_\Ws(u)_\la B(v) - \rep_\Ws(v)_{-\la-\p}B(u),\\
  & \{u_\la v\}_\Ws + \rep_\Vs(B(u))_\la v
             - \rep_\Vs(B(v))_{-\la-\p}u
             + \psi_{1,\la}(B(u),B(v))) \\
      =& (0,\{u_\la v\}_\Ws + \rep_\Vs(B(u))_\la v
             - \rep_\Vs(B(v))_{-\la-\p}u
             + \psi_{1,\la}(B(u),B(v))) \\ 
=& (0,\{u_\la v\}^B_\Ws).     
\end{align*}
Since $\Omega_\la$ defines a Lie conformal bracket then $\{\cdot_\la \cdot\}^D_\Vs$ so is. Hence $(\Ws,\{\cdot_\la \cdot\}^B_\Ws)$ is a Lie conformal algebra.

\end{proof}

\begin{theorem}
Let $(\Ep,\Vs,\Ws)$ be a quasi-twilled Lie conformal algebra and $B: \Ws \to \Vs$ be a left deformation map. Define the map $\rep^B_\Ws: \Ws \to Cend(\Vs)$ by
\begin{align}
 \rep^B_\Ws(u)_\la a
  =& \{B(u)_\la a\}_\Vs + \rep_\Ws(u)_\la a +B(\rep_\Vs(a)_{-\la-\p}u)-B(\psi_{1,\la}(B(u),a)).
 \label{eq:induced-rep-II}    
\end{align}
Then the pair $(\Vs, \rep^B_\Ws)$ is a representation of the Lie conformal algebra $\Ws^B$. 
\end{theorem}

\begin{proof}
We have already seen that  $(\Ep,\Vs,\Ws,\Omega^B)$ is a quasi-twilled Lie conformal algebra in which $\Ws$ is a  subalgebra   (and the induced Lie conformal algebra structure is $\Ws^B$).     Moreover, it follows that
the Lie conformal  algebra $\Ws^B$ has a representation on the $\mathbb{C}[\partial]$-module $\Vs$ with the action map given by
\begin{align*}
&pr_{\Vs}(\Omega^B_\la((0,u),(a,0))\\=& pr_{\Vs}((I-\Bar{B})(\Omega_\la((B(u),u),(a,0))) \\
=&  pr_{\Ws}((I-\Bar{D}) \bigl(
\{B(u)_\la a\}_\Vs + \rep_\Ws(u)_\la a,
- \rep_\Vs(a)_{-\la-\p}u + \psi_{1,\la}(B(u),a)
\bigr)\\
=&(\{B(u)_\la a\}_\Vs + \rep_\Ws(u)_\la a +B(\rep_\Vs(a)_{-\la-\p}u)-B(\psi_{1,\la}(B(u),a)), - \rep_\Vs(a)_{-\la-\p}u + \psi_{1,\la}(B(u),a) )   \\
=&  \{B(u)_\la a\}_\Vs + \rep_\Ws(u)_\la a +B(\rep_\Vs(a)_{-\la-\p}u)-B(\psi_{1,\la}(B(u),a)).           
\end{align*}
This completes the proof. 
\end{proof}

Let $\delta^B:\Cch^k(\Ws,\Vs)\to\Cch^{k+1}(\Ws,\Vs)$ be the
Chevalley-Eilenberg coboundary of $(\Ws,\{\cdot_\la\cdot\}^B_\Ws)$
with coefficients in $(\Vs; \rep^B_\Ws)$.  For $f\in\Cch^k(\Ws,\Vs)$ and
$u_1,\ldots,u_{k+1}\in\Ws$:
\begin{equation}\label{eq:CE-B-II}
\begin{split}
  (\delta^B f)_{\la_1,\ldots,\la_k}
  &(u_1,\ldots,u_{k+1}) \\
  &= \sum_{i=1}^{k}(-1)^{i+1}
     \rep^B_\Ws(u_i)_{\la_i}
     f_{\la_1,\ldots,\widehat{\la_i},\ldots \la_k}(u_1,\ldots,\widehat{u_i},\ldots,u_{k+1})
     %\big|_{\la_{k+1}\mapsto\la_{k+1}^\ddag} 
     \\
  &\quad
    + \sum_{i=1}^{k}(-1)^i
    f_{\la_1,\ldots,\widehat{\la_i},\ldots \la_k}
      (u_1,\ldots,\widehat{u_i},\ldots,u_k,
       \{u_{i\,\la_i} u_{k+1}\}^B_\Ws)
      %\big|_{\la_{k+1}\mapsto\la_{k+1}^\ddag} 
      \\
  &\quad
    + \!\!\sum_{1\leq i<j\leq k+1}\!\!(-1)^{k+i+j+1}
      f_{\la_1,\ldots,\widehat{\la_i},\ldots ,\widehat{\la_j},\ldots, \la_k,\la_{k+1}^\ddag}
      (u_1,\ldots,u_k,\{u_{i\,\la_i}u_j\}^B_\Ws)
      %\big|_{\la_{k+1}\mapsto\la_{k+1}^\ddag} 
      \\
  &\quad
    + (-1)^k\rep^B_\Ws(u_{k+1})_{\la_{k+1}^\ddag}
      f_{\la_1,\ldots,\la_{k-1}}(u_1,\ldots,u_k)
.
\end{split}
\end{equation}
Expanding via Eqs. \eqref{eq:induced-bracket-II}-\eqref{eq:induced-rep-II}
gives, for all $u_1,\ldots,u_{k+1}\in\Ws$:

\begin{align*}
 & (\delta^B f)_{\la_1,\ldots,\la_k}
  (u_1,\ldots,u_{k+1}) \\
 = & \sum_{i=1}^{k}(-1)^{i+1}
     \rep_\Ws(u_i)_{\la_i}
     f_{\la_1,\ldots,\widehat{\la_i},\ldots \la_k}(u_1,\ldots,\widehat{u_i},\ldots,u_{k+1})
     %\big|_{\la_{k+1}\mapsto\la_{k+1}^\ddag} 
     \\
  &+\sum_{i=1}^{k}(-1)^{i+1}
     \{B(u_i)_{\la_i}
     f_{\la_1,\ldots,\widehat{\la_i},\ldots \la_k}(u_1,\ldots,\widehat{u_i},\ldots,u_{k+1})\}_\Vs
     %\big|_{\la_{k+1}\mapsto\la_{k+1}^\ddag} 
     \\
  &+\sum_{i=1}^{k}(-1)^{i+1}
    B (\rep_\Vs(f_{\la_1,\ldots,\widehat{\la_i},\ldots \la_k}(u_1,\ldots,\widehat{u_i},\ldots,u_{k+1}))_{\la_i}u_i
%\big|_{\la_{k+1}\mapsto\la_{k+1}^\ddag} 
     \\
  &-\sum_{i=1}^{k}(-1)^{i+1}
    B( \psi_{1,\la_i}(B(u_i),
     f_{\la_1,\ldots,\widehat{\la_i},\ldots \la_k}(u_1,\ldots,\widehat{u_i},\ldots,u_{k+1}))
     %\big|_{\la_{k+1}\mapsto\la_{k+1}^\ddag} 
     \\
  &
    + \sum_{i=1}^{k}(-1)^i
    f_{\la_1,\ldots,\widehat{\la_i},\ldots \la_k}
      (u_1,\ldots,\widehat{u_i},\ldots,u_k,
       \{u_{i\,\la_i} u_{k+1}\}_\Ws)
      %\big|_{\la_{k+1}\mapsto\la_{k+1}^\ddag} 
      \\&
    + \sum_{i=1}^{k}(-1)^i
    f_{\la_1,\ldots,\widehat{\la_i},\ldots \la_k}
      (u_1,\ldots,\widehat{u_i},\ldots,u_k,
      \psi_{1,\la_i} (u_{i}, u_{k+1}))
      %\big|_{\la_{k+1}\mapsto\la_{k+1}^\ddag} 
      \\&
    + \sum_{i=1}^{k}(-1)^i
    f_{\la_1,\ldots,\widehat{\la_i},\ldots \la_k}
      (u_1,\ldots,\widehat{u_i},\ldots,u_k,
       \rep_\Vs(B(u_i))_{\la_i}u_{k+1}-\rep_\Vs(B(u_{k+1}))_{-\la_i-\partial}u_i)
      %\big|_{\la_{k+1}\mapsto\la_{k+1}^\ddag} 
      \\
  &
    + \!\!\sum_{1\leq i<j\leq k}\!\!(-1)^{k+i+j+1}
      f_{\la_1,\ldots,\widehat{\la_i},\ldots ,\widehat{\la_j},\ldots, \la_k,\la_{k+1}^\ddag}
      (u_1,\ldots,u_k,\{u_{i\,\la_i}u_j\}_\Ws)\\
  &
    + \!\!\sum_{1\leq i<j\leq k}\!\!(-1)^{k+i+j+1}
      f_{\la_1,\ldots,\widehat{\la_i},\ldots ,\widehat{\la_j},\ldots, \la_k,\la_{k+1}^\ddag}
      (u_1,\ldots,u_k,\psi_{1,\la_i} (u_{i}, u_{k+1}))\\
  &
    + \!\!\sum_{1\leq i<j\leq k}\!\!(-1)^{k+i+j+1}
      f_{\la_1,\ldots,\widehat{\la_i},\ldots ,\widehat{\la_j},\ldots, \la_k,\la_{k+1}^\ddag}
    (u_1,\ldots,u_k,\rep_\Vs(B(u_i))_{\la_i}u_j-\rep_\Vs(B(u_j))_{-\la_i-\partial}u_i)
      %\big|_{\la_{k+1}\mapsto\la_{k+1}^\ddag} 
      \\&+
      (-1)^{k}
     \rep_\Ws(u_i)_{k+1})_{\la_{k+1}^\ddag}
      f_{\la_1,\ldots,\la_{k-1}}(u_1,\ldots,u_{k})
     %\big|_{\la_{k+1}\mapsto\la_{k+1}^\ddag} 
     +(-1)^{k}
     \{B(u_{k+1})_{\la_{k+1}^\ddag}
      f_{\la_1,\ldots,\la_{k-1}}(u_1,\ldots,u_{k+})\}_\Vs
     %\big|_{\la_{k+1}\mapsto\la_{k+1}^\ddag} 
     \\
  &+(-1)^{k}
    B (\rep_\Vs(
      f_{\la_1,\ldots,\la_{k-1}}(u_1,\ldots,u_{k}))_{\la_{k+1}^\ddag}u_i
%\big|_{\la_{k+1}\mapsto\la_{k+1}^\ddag} 
     -(-1)^{k}
    B( \psi_{1,\la_{k+1}^\ddag}(B(u_{k+1}),
     f_{\la_1,\ldots,\la_{k-1}}(u_1,\ldots,u_{k}))
     %\big|_{\la_{k+1}\mapsto\la_{k+1}^\ddag} 
.
\end{align*}

\begin{definition}\label{def:cohom-II}
Let $B:\Ws\to\Vs$ be a left deformation map\ of $(\Ep,\Vs,\Ws)$.  Set $C^0(B):=0$,
$C^1(B):=\Vs$, and $C^n(B):=\Cch^{n-1}(\Ws,\Vs)$ for $n\geq2$.
The \textit{cohomology of the left deformation map\ $B$} is the cohomology of the
cochain complex $\bigl(\bigoplus_{i\geq0}C^i(B),\,\delta^B\bigr)$, with
cohomology groups $H^n(B)$ for $n\geq0$.
\end{definition}
 
An element $a\in\Vs=C^1(B)$ is a $1$-cocycle if and only if
\[
  -\rep_\Ws(u)_\la a + \{B(u)_\la a\}_\Vs
  + B(\rep_\Vs(a)_\la u) - B(\psi_{1,\la}(B(u),a)) = 0,
  \quad \forall\,u\in\Ws,
\]
and $f\in\Cch^1(\Ws,\Vs)=C^2(B)$ is a $2$-cocycle if and only if
\begin{align*}
  &-\rep_\Ws(u)_\la f(v)+\rep_\Ws(v)_{-\la-\p}f(u)
  +\{B(u)_\la f(v)\}_\Vs-\{B(v)_{-\la-\p}f(u)\}_\Vs\\
  &\quad +B(\rep_\Vs(f(v))_\la u)-B(\rep_\Vs(f(u))_{-\la-\p}v)
  -B(\psi_{1,\la}(B(u),f(v)))+B(\psi_{1,\la}(B(v),f(u)))\\
  &= f(\{u_\la v\}_\Ws+\psi_{1,\la}(B(u),B(v)))
     +f(\rep_\Vs(B(u))_\la v-\rep_\Vs(B(v))_{-\la-\p}u),
\end{align*}
for all $u,v\in\Ws$.
 
%\begin{proposition}\label{prop:l1B-CE-II}
%For any $f\in\Cch^k(\Ws,\Vs)$, one has
%$l_1^B(f) = (-1)^{k-1}\delta^B f.$
%In particular, $H^1(B)$ classifies infinitesimal deformations of $B$.
%\end{proposition}
 
%\begin{proof}
%By \eqref{eq:l1B-II} and Theorem~\ref{thm:control-II},
%$l_1^B(f) = [\smv,\widehat{f}]_{NR}
%+ [[\smu,\widehat{B}]_{NR},\widehat{f}]_{NR}
%+ \frac{1}{2}[[[\shf,\widehat{B}]_{NR},\widehat{B}]_{NR},\widehat{f}]_{NR}
%= [\smv^B,\widehat{f}]_{NR} = (-1)^{k-1}\delta^B f$.
%\end{proof}
 
Definition~\ref{def:cohom-II} unifies the cohomologies of  special cases.
 
\begin{example}[Cohomology of a relative Rota-Baxter operator]\label{ex:coh conf-RRB}
In the setting of Example~\ref{ex:conf-RRB}, the induced bracket is
\begin{align*}
 \{u_\la v\}^B_\Ws = \la\{u_\la v\}_\Ws + \rep_\Vs(B(u))_\la v
- \rep_\Vs(B(v))_{-\la-\p}u,   
\end{align*}
and the module action is
\begin{align*}
 \rep_{\Ws}^B(u)_\la a = \{B(u)_\la a\}_\Vs + B(\rep_\Vs(a)_\la u).  
\end{align*}
The cohomology is given in \cite{YL}.
\end{example}
 
\begin{example}[Cohomology of a conformal twisted Rota-Baxter operator]
In the setting of Example~\ref{ex:conf-TRB}, the induced bracket is
$\{u_\la v\}^B_\Mod = \rep_\Vs(B(u))_\la v - \rep_\Vs(B(u))_{-\la-\p}v
+ \phi_\la(B(u),B(v))$, and the module action is
$\rep_\Ws^B(u)_\la a = \{B(u)_\la a\}_\Vs + B(\rep_\Vs(a)_\la u)
- B(\phi_\la(B(u),a))$.
The cohomology of Definition~\ref{def:cohom-II} is the
\textit{cohomology for a conformal twisted Rota-Baxter operator} \cite{YL}.
\end{example}
 
\begin{example}[Cohomology of a conformal Reynolds operator]
In the setting of Example~\ref{ex:conf-Reynolds}, the induced bracket
is $\{u_\la v\}^B_\Vs = \{B(u)_\la v\}_\Vs - \{B(v)_{-\la-\p}u\}_\Vs
+ \{B(u)_\la B(v)\}_\Vs$, and the module action is
$\rep_\Ws^B(u)_\la a = \{B(u)_\la a\}_\Vs + B(\{a_\la u\}_\Vs)
- B(\{B(u)_\la a\}_\Vs)$.
The cohomology of Definition~\ref{def:cohom-II} is the
\textit{cohomology for a conformal Reynolds operator} \cite{YL}.
\end{example}
 
\begin{example}\label{def:cohom-matched}
Consider Example~\ref{ex:conf-matched-pair}.  Let $B:\Ws\to\Vs$ be a
deformation map of the matched pair $(\Vs,\Ws;\rep_\Vs,\rep_\Ws)$ of
Lie conformal  algebras.  The \textit{cohomology for the deformation map
$B$ of a matched pair of Lie conformal  algebras} is the
Chevalley-Eilenberg cohomology of $(\Ws,\{\cdot_\la\cdot\}^B_\Ws)$
with coefficients in the module $(\Vs;\rep_\Ws^B)$, where
\[
  \rep_\Ws^B(u)_\la a
  = \rep_\Ws(u)_\la a + \{B(u)_\la a\}_\Vs + B(\rep_\Vs(a)_\la u).
\]
The coboundary operators are defined as,  for all $u_1,\ldots,u_{k+1}\in\Ws$:

\begin{align}\label{Coh: cohom-matched}
\begin{split}
 & (\delta^B f)_{\la_1,\ldots,\la_k}
  (u_1,\ldots,u_{k+1}) \\
 = & \sum_{i=1}^{k}(-1)^{i+1}
     \rep_\Ws(u_i)_{\la_i}
     f_{\la_1,\ldots,\widehat{\la_i},\ldots \la_k}(u_1,\ldots,\widehat{u_i},\ldots,u_{k+1})
     %\big|_{\la_{k+1}\mapsto\la_{k+1}^\ddag} 
     \\
  &+\sum_{i=1}^{k}(-1)^{i+1}
     \{B(u_i)_{\la_i}
     f_{\la_1,\ldots,\widehat{\la_i},\ldots \la_k}(u_1,\ldots,\widehat{u_i},\ldots,u_{k+1})\}_\Vs
     %\big|_{\la_{k+1}\mapsto\la_{k+1}^\ddag} 
     \\
  &+\sum_{i=1}^{k}(-1)^{i+1}
    B (\rep_\Vs(f_{\la_1,\ldots,\widehat{\la_i},\ldots \la_k}(u_1,\ldots,\widehat{u_i},\ldots,u_{k+1}))_{\la_i}u_i
%\big|_{\la_{k+1}\mapsto\la_{k+1}^\ddag} 
     \\
  &
    + \sum_{i=1}^{k}(-1)^i
    f_{\la_1,\ldots,\widehat{\la_i},\ldots \la_k}
      (u_1,\ldots,\widehat{u_i},\ldots,u_k,
       \{u_{i\,\la_i} u_{k+1}\}_\Ws)
      %\big|_{\la_{k+1}\mapsto\la_{k+1}^\ddag} 
      \\&
    + \sum_{i=1}^{k}(-1)^i
    f_{\la_1,\ldots,\widehat{\la_i},\ldots \la_k}
      (u_1,\ldots,\widehat{u_i},\ldots,u_k,
      \psi_{1,\la_i} (u_{i}, u_{k+1}))
      %\big|_{\la_{k+1}\mapsto\la_{k+1}^\ddag} 
      \\&
    + \sum_{i=1}^{k}(-1)^i
    f_{\la_1,\ldots,\widehat{\la_i},\ldots \la_k}
      (u_1,\ldots,\widehat{u_i},\ldots,u_k,
       \rep_\Vs(B(u_i))_{\la_i}u_{k+1}-\rep_\Vs(B(u_{k+1}))_{-\la_i-\partial}u_i)
      %\big|_{\la_{k+1}\mapsto\la_{k+1}^\ddag} 
      \\
  &
    + \!\!\sum_{1\leq i<j\leq k}\!\!(-1)^{k+i+j+1}
      f_{\la_1,\ldots,\widehat{\la_i},\ldots ,\widehat{\la_j},\ldots, \la_k,\la_{k+1}^\ddag}
      (u_1,\ldots,u_k,\{u_{i\,\la_i}u_j\}_\Ws)\\
  &
    + \!\!\sum_{1\leq i<j\leq k}\!\!(-1)^{k+i+j+1}
      f_{\la_1,\ldots,\widehat{\la_i},\ldots ,\widehat{\la_j},\ldots, \la_k,\la_{k+1}^\ddag}
      (u_1,\ldots,u_k,\psi_{1,\la_i} (u_{i}, u_{k+1}))\\
  &
    + \!\!\sum_{1\leq i<j\leq k}\!\!(-1)^{k+i+j+1}
      f_{\la_1,\ldots,\widehat{\la_i},\ldots ,\widehat{\la_j},\ldots, \la_k,\la_{k+1}^\ddag}
    (u_1,\ldots,u_k,\rep_\Vs(B(u_i))_{\la_i}u_j-\rep_\Vs(B(u_j))_{-\la_i-\partial}u_i)
      %\big|_{\la_{k+1}\mapsto\la_{k+1}^\ddag} 
      \\&+
      (-1)^{k}
     \rep_\Ws(u_i)_{k+1})_{\la_{k+1}^\ddag}
      f_{\la_1,\ldots,\la_{k-1}}(u_1,\ldots,u_{k})
     %\big|_{\la_{k+1}\mapsto\la_{k+1}^\ddag} 
     +(-1)^{k}
     \{B(u_{k+1})_{\la_{k+1}^\ddag}
      f_{\la_1,\ldots,\la_{k-1}}(u_1,\ldots,u_{k+})\}_\Vs
     %\big|_{\la_{k+1}\mapsto\la_{k+1}^\ddag} 
     \\
  &+(-1)^{k}
    B (\rep_\Vs(
      f_{\la_1,\ldots,\la_{k-1}}(u_1,\ldots,u_{k}))_{\la_{k+1}^\ddag}u_i
%\big|_{\la_{k+1}\mapsto\la_{k+1}^\ddag} 
.\end{split}
\end{align}
So far, this result provides a new cohomology.
\end{example}
 
\begin{remark}
In the matched pair case, the Lie conformal  algebra structure on $\Ws$
induced by $B$ can equivalently be obtained by transferring the
Lie conformal  algebra structure on $\mathrm{Gr}(B)$ to $\Ws$ via the
$\Rg$-module isomorphism $u\mapsto(B(u),u)$.
\end{remark}
 
\begin{remark}
The second cohomology group $H^2(B)$ in Definition~\ref{def:cohom-II}
(resp.\ Definition~\ref{def:cohom-matched}) classifies and controls
infinitesimal deformations of the left deformation map\ $B$, extending to the
conformal setting the results of \cite{Das0,Das1,JSZ,TBGS}.
\end{remark}

%--------------------------------------------------------------------------
\subsection{The controlling $L_\infty$-algebra of left deformation maps}
\label{subsec:control-II}
%--------------------------------------------------------------------------
 
\begin{theorem}\label{thm:control-II}
Let $(\Ep,\Vs,\Ws)$ be a quasi-twilled Lie conformal  algebra with
$\str=\shf+\smu+\smv$.  Set
\[
  L := \bigoplus_{n\geq0}\Cch^{n+1}(\Ep,\Ep), \qquad
  F := \bigoplus_{n\geq0}\Cch^{n+1}(\Ws,\Vs),
\]
let $\mathbf{P}:L\to L$ be the projection onto $F$, and set
$\Delta:=\str$.  Then $(L,F,\mathbf{P},\Delta)$ is a $V$-data, and
the resulting $L_\infty$-algebra on $F$ is
\[
  \Bigl(\bigoplus_{n\geq0}\Cch^{n+1}(\Ws,\Vs),\; l_1,\, l_2,\, l_3\Bigr),
\]
with
\begin{align}
  l_1(f)
  &= [\smv,\,\widehat{f}\,]_{NR},
  \label{eq:l1-II}\\
  l_2(f_1,f_2)
  &= [[\smu,\,\widehat{f}_1]_{NR},\,\widehat{f}_2]_{NR},
  \label{eq:l2-II}\\
  l_3(f_1,f_2,f_3)
  &= [[[\shf,\,\widehat{f}_1]_{NR},\,\widehat{f}_2]_{NR},\,\widehat{f}_3]_{NR},
  \label{eq:l3-II}\\
  l_k &= 0, \quad k\geq4.
\end{align}
A $\Rg$-module homomorphism $B:\Ws\to\Vs$ is a left deformation map\ of
$(\Ep,\Vs,\Ws)$ if and only if $B$ is a Maurer-Cartan element:
\begin{equation}\label{eq:MC-II}
  l_1(B) + \tfrac{1}{2}l_2(B,B) + \tfrac{1}{6}l_3(B,B,B) = 0.
\end{equation}
\end{theorem}
 
\begin{proof}
Since $\str\in\ker(\mathbf{P})^1$ (as $\str$ has no component in
$F=\bigoplus\Cch^{n+1}(\Ws,\Vs)$) and $[\str,\str]_{NR}=0$, the data
$(L,F,\mathbf{P},\str)$ is a $V$-data.  By Theorem~\ref{thm:Voronov}:
\begin{align*}
  l_1(f) &= \mathbf{P}([\str,\widehat{f}]_{NR}) = [\smv,\widehat{f}]_{NR},\\
  l_2(f_1,f_2) &= \mathbf{P}([[\str,\widehat{f}_1]_{NR},\widehat{f}_2]_{NR})
               = [[\smu,\widehat{f}_1]_{NR},\widehat{f}_2]_{NR},\\
  l_3(f_1,f_2,f_3) &= \mathbf{P}([[[\str,\widehat{f}_1]_{NR},\widehat{f}_2]_{NR},
               \widehat{f}_3]_{NR})
               = [[[\shf,\widehat{f}_1]_{NR},\widehat{f}_2]_{NR},\widehat{f}_3]_{NR},
\end{align*}
and $l_k=0$ for $k\geq4$ by abelianness of $F$.  Expanding
\eqref{eq:MC-II} via \eqref{eq:l1-II}-\eqref{eq:l3-II} gives exactly
$\psi^B_{2,\la}(u,v)=0$, which is \eqref{eq:DmapII}.
\end{proof}
 
\begin{corollary}[Controlling algebra for weighted relative Rota-Baxter operators]
\label{cor:control-RRB}
In the setting of Example~\ref{ex:conf-RRB}, the $L_\infty$-algebra
\eqref{eq:l1-II}-\eqref{eq:l3-II} degenerates ($\shf=0$) to a
differential graded Lie algebra
$\bigl(\bigoplus_{n\geq1}\Cch^n(\Ws,\Vs),\,\dM,\,\langle\cdot,\cdot\rangle\bigr)$,
where the differential $\dM:\Cch^p(\Ws,\Vs)\to\Cch^{p+1}(\Ws,\Vs)$ is
\begin{equation}\label{eq:diff-RRB}
  (\dM f)_{\la_1,\ldots,\la_p}(u_1,\ldots,u_{p+1})
  = \sum_{i<j}(-1)^{p+i+j-1}\la\,
    f_{\la_i+\la_j,\widehat{\la_i},\widehat{\la_j}}
    (\{u_i{}_{\la_i}u_j\}_\Ws,\ldots)
    \big|_{\la_{p+1}^\ddag},
\end{equation}
and the graded bracket $\langle\cdot,\cdot\rangle$ is given by
\begin{equation}\label{eq:bracket-RRB}
\begin{split}
  \langle f_1,f_2\rangle_{\boldsymbol\la}(u_1,\ldots,u_{p+q})
  &= -\!\!\sum_{\sigma\in\perm_{(q,1,p-1)}}(-1)^\sigma
     f_1\!\bigl(\rep(f_2(u_{\sigma(1)},\ldots,u_{\sigma(q)}))_{\la_{\sigma(q+1)}}
     u_{\sigma(q+1)},\ldots\bigr)\big|_{\la_{p+q}^\ddag}\\
  &\quad + (-1)^{pq}\!\!\sum_{\sigma\in\perm_{(p,1,q-1)}}(-1)^\sigma
     f_2\!\bigl(\rep(f_1(\ldots))_{\la_{\sigma(p+1)}}
     u_{\sigma(p+1)},\ldots\bigr)\big|_{\la_{p+q}^\ddag}\\
  &\quad - (-1)^{pq}\!\!\sum_{\sigma\in\perm_{(p,q)}}(-1)^\sigma
     \{f_1(\ldots)_\la f_2(\ldots)\}_\Vs\big|_{\la_{p+q}^\ddag},
\end{split}
\end{equation}
for all $f_1\in\Cch^p(\Ws,\Vs)$, $f_2\in\Cch^q(\Ws,\Vs)$.
This is the \textit{controlling algebra for conformal relative
Rota-Baxter operators of weight $\la$}.
\end{corollary}
 
\begin{corollary}[Controlling algebra for conformal $\mathcal{O}$-operators]
In the setting of Example~\ref{ex:conf-RRB} with $\la=0$,
$\bigl(\bigoplus_{n\geq1}\Cch^n(\Ws,\Vs),\,\langle\cdot,\cdot\rangle\bigr)$
is a graded Lie algebra.  This is the \textit{controlling algebra for
conformal $\mathcal{O}$-operators}.
\end{corollary}
 
\begin{corollary}[Controlling algebra for conformal twisted Rota-Baxter operators]
\label{cor:control-TRB}
In the setting of Example~\ref{ex:conf-TRB},
$\bigl(\bigoplus_{n\geq0}\Cch^{n+1}(\Mod,\Vs),\,l_2,\,l_3\bigr)$
is an $L_\infty$-algebra where $l_2$ and $l_3$ are given by
\eqref{eq:l2-II}-\eqref{eq:l3-II} with $\shf=\widehat\phi$.
This is the \textit{controlling algebra for conformal twisted
Rota-Baxter operators}.
\end{corollary}
 
\begin{corollary}[Controlling algebra for conformal Reynolds operators]
\label{cor:control-Reynolds}
In the setting of Example~\ref{ex:conf-Reynolds},
$\bigl(\bigoplus_{n\geq0}\Cch^{n+1}(\Vs,\Vs),\,l_2,\,l_3\bigr)$
is an $L_\infty$-algebra.  This is the \textit{controlling algebra for
conformal Reynolds operators}.
\end{corollary}
 
Theorem~\ref{thm:control-II} also gives a new result:
 
\begin{corollary}[Controlling algebra for deformation maps of matched pairs]
\label{cor:control-matched}
In the setting of Example~\ref{ex:conf-matched-pair} ($\shf=0$), the
controlling structure is a differential graded Lie algebra
$\bigl(\bigoplus_{n\geq1}\Cch^n(\Ws,\Vs),\,\dM,\,\langle\cdot,\cdot\rangle\bigr)$,
where $\langle\cdot,\cdot\rangle$ is given by \eqref{eq:bracket-RRB} and
\begin{equation}\label{eq:diff-matched}
\begin{split}
  (\dM f)_{\la_1,\ldots,\la_p}(u_1,\ldots,u_{p+1})
  &= \sum_{i=1}^{p+1}(-1)^{p+i}
     \rep_\Ws(u_i)_{\la_i}
     f_{\widehat{\la_i}}(\ldots)\big|_{\la_{p+1}^\ddag}\\
  &\quad + \sum_{i<j}(-1)^{p+i+j-1}
     f_{\la_i+\la_j,\widehat{\la_i},\widehat{\la_j}}
     (\{u_i{}_{\la_i}u_j\}_\Ws,\ldots)\big|_{\la_{p+1}^\ddag}.
\end{split}
\end{equation}
Maurer-Cartan elements are exactly deformation maps of a matched pair
of Lie conformal  algebras.
\end{corollary}
 
Let $B:\Ws\to\Vs$ be a left deformation map.  The twisted $L_\infty$-algebra is given by 
\begin{align}
  l_1^B(f)
  &= l_1(f) + l_2(B,f) + \tfrac{1}{2}l_3(B,B,f),
  \label{eq:l1B-II}\\
  l_2^B(f_1,f_2)
  &= l_2(f_1,f_2) + l_3(B,f_1,f_2),
  \label{eq:l2B-II}\\
  l_3^B(f_1,f_2,f_3)
  &= l_3(f_1,f_2,f_3),
  \label{eq:l3B-II}\\
  l_k^B &= 0, \quad k\geq4.
\end{align}
 
\begin{theorem}\label{thm:deform-DmapII}
Let $B:\Ws\to\Vs$ be a left deformation map\ of $(\Ep,\Vs,\Ws)$.  A $\Rg$-module
homomorphism $B':\Ws\to\Vs$ gives a left deformation map\ $B+B'$ if and only if $B'$
is a Maurer-Cartan element of the twisted $L_\infty$-algebra
$\bigl(\bigoplus_{n\geq0}\Cch^{n+1}(\Ws,\Vs),\,l_1^B,\,l_2^B,\,l_3^B\bigr)$, that is 
\begin{equation}\label{eq:MC-deform-II}
  l_1^B(B') + \tfrac{1}{2}l_2^B(B',B') + \tfrac{1}{6}l_3^B(B',B',B') = 0.
\end{equation}
\end{theorem}
 
\begin{proof}
By Theorem~\ref{thm:control-II}, $B+B'$ is a left deformation map if and only if
$\sum_{k=1}^3\frac{1}{k!}l_k(B+B',\ldots)=0$.  Using the Maurer-Cartan  equation
for $B$ gives \eqref{eq:MC-deform-II}.
\end{proof}
 
Applying Theorem~\ref{thm:deform-DmapII} to
Corollary~\ref{cor:control-matched}, the differential graded Lie algebra
governing deformations of a deformation map $B:\Ws\to\Vs$ of a matched
pair has differential

\begin{equation}\label{eq:dB-matched}
\begin{split}
  (\dM^B f)_{\boldsymbol\la}(u_1,\ldots,u_{p+1})
  &= \sum_{i=1}^{p+1}(-1)^{p+i}
     \rep_\Ws(u_i)_{\la_i}
     f_{\widehat{\la_i}}(\ldots)\big|_{\la_{p+1}^\ddag}
  + \sum_{i<j}(\cdots)\big|_{\la_{p+1}^\ddag}\\
  &\quad + \sum_{i=1}^{p+1}(-1)^{p+i}
     B(\rep_\Vs(f(\ldots))_{\la_i}u_i)\big|_{\la_{p+1}^\ddag}
  - \sum_{i=1}^{p+1}(-1)^{p+i}
     f(\rep_\Vs(B(u_i)),\ldots)\big|_{\la_{p+1}^\ddag}\\
  &\quad + \sum_{i=1}^{p+1}(-1)^{p+i}
     \{B(u_i)_{\la_i}f(\ldots)\}_\Vs\big|_{\la_{p+1}^\ddag}.
\end{split}
\end{equation}
 
\begin{remark}
Applying Theorem~\ref{thm:deform-DmapII} to
Corollaries~\ref{cor:control-RRB}-\ref{cor:control-Reynolds} recovers
the algebras governing deformations of conformal relative Rota-Baxter
operators, conformal twisted Rota-Baxter operators, and conformal
Reynolds operators, extending the classical results of \cite{Das0,Das1,TBGS}
to the conformal setting. The complete study of deformations will be provided in a forthcoming paper.
\end{remark}
 \section*{Summary} The following tables summarize the Lie conformal algebras structure and representation  induced by special cases of right and left deformation maps on a quasi-twilled Lie conformal algebra $(\Ep,\Vs,\Ws)$, together with their corresponding representations and cohomologies.
Table~\ref{tab1} presents the special cases of right deformation maps, including conformal modified $r$-matrices, conformal crossed homomorphisms, conformal derivations, and Lie conformal  algebra homomorphisms.
Table~\ref{tab2} lists the special cases of left deformation maps, including conformal relative Rota--Baxter operators, conformal twisted Rota-Baxter operators, conformal Reynolds operators, and deformation maps arising from matched pairs.
\begin{table}[ht]
\centering
\caption{Right deformation map $D:\Vs\to\Ws$}
\label{tab1}
\begin{tabular}{|c|c|c|}
\hline
  conformal operator & Induced Lie conformal  structure&  Induced representation\\
\hline
  $r$-matrix & $\{a_\la D(b)\}_\Vs - \{b_{-\la-\p}D(a)\}_\Vs$&  $\{D(a)_\la u\}_\Vs - D(\{u_\la a\}_\Vs)$\\
 \hline
 crossed homomorphism & $\{a_\la b\}_\Vs $&  $\rep(a)_\la u + \{D(a)_\la u\}_\Ws)$\\
 \hline
 conformal derivation & $\{a_\la b\}_\Vs$&  $\rep_\Vs(a)_\la u $\\
 \hline
 conf. Lie alg. homomorphism & $\{a_\la b\}_\Vs $&  $\{D(a)_\la u\}_\Ws)$\\
 \hline
\end{tabular}
\end{table}

In the classical case, the cohomology theories have been extensively studied in the literature,  see for   modified $r$-matrices \cite{JS2}, crossed homomorphisms \cite{PSTZ}, derivations \cite{TFS} and Lie algebra homomorphisms \cite{Fregier1,Fregier2}. In the table above, we summarize the Lie conformal algebras induced by these operators, together with their corresponding representations. For each of these operators, the associated cohomology theory is defined in terms of the Chevalley--Eilenberg cohomology of a suitable Lie conformal algebra induced by the operator, with coefficients in an appropriate induced representation. More precisely, the corresponding cohomology theories are described in Example~\ref{ex:Coh modified-r} for modified $r$-matrices, Example~\ref{ex:coh crossed-hom} for crossed homomorphisms, Example~\ref{ex:coh conf-derivation} for derivations, and Example~\ref{ex:coh conf-hom} for Lie conformal algebra homomorphisms. It turns out that these cohomologies are new except for 
Example~\ref{ex:coh crossed-hom}, see \cite{BBSS}.

\begin{table}[ht]
\centering
\caption{Left deformation map $B:\Ws\to\Vs$}
\label{tab2}
\begin{tabular}{|c|c|c|}
\hline
 conformal operator & Induced Lie conformal  structure&  Induced representation\\
\hline
  relative Rota-Baxter operator & $\{u_\la v\}_\Ws + \rep_\Vs(B(u))_\la v
$&  $\{B(u)_\la a\}_\Vs + B(\rep_\Vs(a)_\la u$\\
   & $
- \rep_\Vs(B(v))_{-\la-\p}u$&  \\\hline
  twisted Rota-Baxter operator& $\rep_\Vs(B(u))_\la v - \rep_\Vs(B(u))_{-\la-\p}v
$&  $\{B(u)_\la a\}_\Vs + B(\rep_\Vs(a)_\la u)
$\\& $
+ \phi_\la(B(u),B(v))$&  $
- B(\phi_\la(B(u),a))$\\\hline
  Reynolds operator & $\{B(u)_\la v\}_\Vs - \{B(v)_{-\la-\p}u\}_\Vs
$&  $\{B(u)_\la a\}_\Vs + B(\{a_\la u\}_\Vs)
$\\
  & $+ \{B(u)_\la B(v)\}_\Vs$&  $
- B(\{B(u)_\la a\}_\Vs)$\\
 \hline Deformation map of a matched pair  & $\{u_\la v\}_\Ws + \rep_\Vs(B(u))_\la v
$  &  $\{B(u)_\la a\}_\Vs + \rep_\Ws(u)_\la a 
$\\
 of Lie conformal  algebras & $- \rep_\Vs(B(v))_{-\la-\p}u$&  $
 +B(\rep_\Vs(a)_{-\la-\p}u)$\\
 \hline
\end{tabular}
\end{table}

Recently, the cohomology theories of Lie conformal algebras associated with relative Rota--Baxter operators of weight $0$, twisted Rota--Baxter operators, and Reynolds operators have been extensively studied in \cite{YL}. In the classical setting, the cohomology theories of Lie algebras associated with relative Rota--Baxter operators (also known as $\mathcal{O}$-operators) and twisted Rota--Baxter operators, which generalize Reynolds operators, were studied in \cite{TBGS} and \cite{Das21}, respectively.
In Table~\ref{tab2}, we list the Lie conformal algebras induced by each of the above conformal operators, together with their corresponding representations. The cohomology of Lie conformal algebras associated with deformation maps of matched pairs is new and defined in Example~\ref{def:cohom-matched}. In general, the cohomology theories associated with these operators are defined as the Chevalley--Eilenberg cohomology of the Lie conformal algebras induced by the respective operators, with coefficients in suitable induced representations.

%==========================================================================


\begin{thebibliography}{999}

\bibitem{AM14a}
A.~L.~Agore and G.~Militaru,
\textit{Extending structures for Lie algebras},
Monatsh. Math. \textbf{174} (2014), no.~2, 169--193.

\bibitem{AM14b}
A.~L.~Agore and G.~Militaru,
\textit{Classifying complements for Hopf algebras and Lie algebras},
 J. Algebra  \textbf{391} (2013), 193--208.
 
\bibitem{BDK}
B.~Bakalov, A.~D'Andrea, and V.~G.~Kac,
\emph{Theory of finite pseudoalgebras},
Adv.\ Math.\ \textbf{162} (2001), 1--140.

\bibitem{BKV}
B.~Bakalov, V.~G.~Kac, and A.~A.~Voronov,
\emph{Cohomology of conformal algebras},
Comm.\ Math.\ Phys.\ \textbf{200} (1999), 561--598.

\bibitem{BBSS} I. Basdouri, S. Benabdelhafidh, A. Saghrouni, S. Silvestrov, 
\emph{Cohomology and deformations of crossed homomorphisms on Lie conformal algebras}, 
Adv. Stud.: Euro-Tbil. Math. J. \textbf{18}, No. 3 (2025), 91--107.

\bibitem{Bax60}
G.~Baxter,
\textit{An analytic problem whose solution follows from a simple algebraic
identity},
Pacific J. Math. \textbf{10} (1960), 731--742.

\bibitem{D'AndreaKac} A. D'Andrea, V.~G.~Kac, \textit{Structure theory of finite conformal algebras,} Sel. Math. \textbf{4} (1998), 377--418

\bibitem{CK98}
A.~Connes and D.~Kreimer,
\textit{Hopf algebras, renormalization and noncommutative geometry},
Comm. Math. Phys. \textbf{199}, no.~1 (1998), 203--242.

\bibitem{CK00}
A.~Connes and D.~Kreimer,
\textit{Renormalization in quantum field theory and the Riemann--Hilbert problem
I: The Hopf algebra structure of graphs and the main theorem},
Comm. Math. Phys. \textbf{210}, no.~1 (2000), 249--273.

\bibitem{Das0}
A.~Das,
\emph{Twisted Rota--Baxter operators and Reynolds operators on Lie
algebras and NS-Lie algebras},
J.\ Math.\ Phys.\ \textbf{62} (2021), Paper~091701.

\bibitem{Das1}
A.~Das,
\emph{Cohomology and deformations of weighted Rota--Baxter operators},
J.\ Math.\ Phys.\ \textbf{63} (2022), 091703.

\bibitem{Das21}
A.~Das,
\textit{Twisted Rota--Baxter operators and Reynolds operators on Lie algebras
and NS-Lie algebras},
J. Math. Phys. \textbf{62} (2021), no.~9, 091701.

\bibitem{das one}
A. Das, S. Majhi and R. Mandal,
\textit{Deformation maps in proto-twilled Leibniz algebras,}
arXiv:2409.18599.

\bibitem{das_two} 
A. Das, 
\emph{Applications of Poisson cohomology to the inducibility problems and study of deformation maps}, 
Journal of Geometry and Physics, (2025): 105704.

\bibitem{DK}
A.~De~Sole and V.~G.~Kac,
\emph{Finite vs.\ affine $W$-algebras},
Japan.\ J.\ Math.\ \textbf{1} (2006), 137--261.

\bibitem{DK3}
A.~De~Sole and V.~G.~Kac,
\emph{The variational Poisson cohomology},
Japan.\ J.\ Math.\ \textbf{8} (2013), 1--145.

\bibitem{DK113}A. De Sole and V.~G.~Kac, \textit{Lie conformal algebra cohomology and the variational complex,} Comm. Math.
Phys. 292 (2009), 667--719. 

\bibitem{DK13}
A.~De~Sole and V.~G.~Kac,
\emph{Essential variational Poisson cohomology},
Comm.\ Math.\ Phys.\ \textbf{313} (2012), 837--864.

\bibitem{DSV}
V.~Dotsenko, S.~Shadrin, and B.~Vallette,
\emph{Maurer--Cartan methods in deformation theory: the twisting procedure},
London Math.\ Soc.\ Lecture Note Ser.\ \textbf{488},
Cambridge University Press, Cambridge, 2023.

\bibitem{Fregier1} Y. Fr\'egier, A new cohomology theory associated to deformations of Lie algebra morphisms. \textit{Lett. Math. Phys}.
\textbf{70} (2004), 97-107. 2, 11, 14

\bibitem{Fregier2} Y. Fr\'egier and M. Zambon, Simultaneous deformations of algebras and morphisms via derived brackets. \textit{J. Pure
Appl. Algebra} \textbf{219} (2015), 5344-5362.

\bibitem{Get}
E.~Getzler,
\emph{Lie theory for nilpotent $L_\infty$-algebras},
Ann.\ Math.\ (2) \textbf{170} (2009), 271--301.

\bibitem{Guo12}
L.~Guo,
\textit{An Introduction to Rota--Baxter Algebra},
Surveys of Modern Mathematics, vol.~4,
International Press / Higher Education Press, 2012.
%\bibitem{SGuo}
%S. Guo, 
% \textit{Cohomology of Lie-conformal algebras with derivations}, 
%J. Math. Res. Appl. 44, No. 6, (2024) 754--768 .

\bibitem{HB}Y. Hong, C.  Bai,  \textit{Conformal classical Yang-Baxter equation, S-equation and O-operators,} Lett. Math. Phys., \textbf{110}, 5 (2019), 885--909.

\bibitem{def_map_3_Lie}
S. Hou, Y. Sheng and Y. Zhou,
\emph{A unified approach to controlling algebras and cohomologies of various operators on $3$-Lie algebras},
Int. J. Geom. Methods Mod. Phys. \textbf{22} (2025), 2540025.

\bibitem{JS2}
J.~Jiang and Y.~Sheng,
\emph{Deformations of modified $r$-matrices and cohomologies of related
algebraic structures},
J.\ Noncommut.\ Geom.\ \textbf{19}, no.~1 (2025), 1--37.

\bibitem{JST23}
J.~Jiang, Y.~Sheng, and R.~Tang,
\textit{Deformation maps in quasi-twilled Lie algebras and their cohomology},
Izvestia: Mathematics, \textbf{90} (3) (2026), 561--589.

\bibitem{JSZ}
J.~Jiang, Y.~Sheng, and C.~Zhu,
\emph{Lie theory and cohomology of relative Rota--Baxter operators},
J.\ Lond.\ Math.\ Soc.\ \textbf{109}, Issue~2 e12836 (2024).

\bibitem{KAC}
V.~G.~Kac,
\emph{Vertex Algebras for Beginners}, 2nd~ed.,
Univ.\ Lecture Ser.\ \textbf{10},
Amer.\ Math.\ Soc., Providence, RI, 1998.

\bibitem{KAC2}  V.~G.~Kac, \emph{Formal distribution algebras and conformal algebras}, Brisbane Congress in Math. Physics,
1997.

\bibitem{KosSch}
Y.~Kosmann-Schwarzbach,
\emph{Derived brackets},
Lett.\ Math.\ Phys.\ \textbf{69} (2004), 61--87.

\bibitem{Ku}
B.~A.~Kupershmidt,
\emph{What a classical $r$-matrix really is},
J.\ Nonlinear Math.\ Phys.\ \textbf{6} (1999), 448--488.

\bibitem{LM}
T.~Lada and M.~Markl,
\emph{Strongly homotopy Lie algebras},
Comm.\ Algebra \textbf{23} (1995), 2147--2161.

\bibitem{LS}
T.~Lada and J.~Stasheff,
\emph{Introduction to $sh$ Lie algebras for physicists},
Internat.\ J.\ Theoret.\ Phys.\ \textbf{32} (1993), 1087--1103.

\bibitem{LST}
A.~Lazarev, Y.~Sheng, and R.~Tang,
\emph{Deformations and homotopy theory of relative Rota--Baxter Lie
algebras},
Comm.\ Math.\ Phys.\ \textbf{383} (2021), 595--631.

\bibitem{Makhlouf}
S. Liu, A. Makhlouf and L. Song,
\emph{Deformation maps of quasi-twilled associative algebras},
J. Algebra, \textbf{689} (2026), 519--550

\bibitem{STS}
M.~A.~Semenov-Tian-Shansky,
\emph{What is a classical $r$-matrix?},
Funct.\ Anal.\ Appl.\ \textbf{17} (1983), 259--272.

\bibitem{STS2}
M.~A.~Semenov-Tian-Shansky,
\emph{Integrable systems and factorization problems},
Oper.\ Theory Adv.\ Appl.\ \textbf{141} (2003), 155--218.
%\bibitem{Sh21}
%Y.~Sheng, ??????
%\textit{Crossed homomorphisms of Lie algebras and their cohomological
%characterization},
%J. Algebra \textbf{574} (2021), 542--568.

\bibitem{Rey1895}
O.~Reynolds,
\textit{On the dynamical theory of incompressible viscous fluids and the
determination of the criterion},
Philos. Trans. R. Soc. London Ser.~A \textbf{186} (1895), 123--164.

\bibitem{Rot69}
G.-C.~Rota,
\textit{Baxter algebras and combinatorial identities. I, II},
Bull. Amer. Math. Soc. \textbf{75} (1969), 325--334.

\bibitem{TBGS}
R.~Tang, C.~Bai, L.~Guo, and Y.~Sheng,
\emph{Deformations and their controlling cohomologies of
$\mathcal{O}$-operators},
Comm.\ Math.\ Phys.\ \textbf{368} (2019), 665--700.

\bibitem{TBGS2}
R.~Tang, C.~Bai, L.~Guo, and Y.~Sheng,
\emph{Homotopy Rota--Baxter operators and post-Lie algebras},
J.\ Noncommut.\ Geom.\ \textbf{17} (2023), 1--35.

\bibitem{TFS}
R.~Tang, Y.~Fr\'egier, and Y.~Sheng,
\emph{Cohomologies of a Lie algebra with a derivation and applications},
J.\ Algebra \textbf{534} (2019), 65--99.

\bibitem{Uch1}
K.~Uchino,
\emph{Derived brackets and sh Leibniz algebras},
J.\ Pure Appl.\ Algebra \textbf{215} (2011), 1102--1111.

\bibitem{Uch08}
K.~Uchino,
\textit{Twisting on associative algebras and Rota--Baxter type operators},
J. Noncommut. Geom. \textbf{2}, no.~2 (2008), 149--170.

\bibitem{PSTZ}  Y. Pei, Y. Sheng, R. Tang and K. Zhao, Actions of monoidal categories and representations of Cartan type Lie
algebras. \textit{J. Inst. Math. Jussieu}, \textbf{22} (2023), 2367-2402

\bibitem{Vo}
T.~T.~Voronov,
\emph{Higher derived brackets and homotopy algebras},
J.\ Pure Appl.\ Algebra \textbf{202} (2005), 133--153.

\bibitem{YL}
L.~Yuan and J.~Liu,
\emph{Twisting theory, relative Rota--Baxter type operators and
$L_\infty$-algebras on Lie conformal algebras},
J.\ Algebra \textbf{636} (2023), 88--122.

\bibitem{Wu}  Z. Wu, \textit{Lie algebra structures on cohomology complexes of some $H$-pseudoalgebras,} J. Algebra \textbf{396} (2013),
117--142.

\bibitem{ZBC} J. Zhao, B. Sun, L. Chen, 
\textit{Representations and cohomology of Rota-Baxter Lie conformal algebras}, 
J. Geom. Phys. \textbf{216} (2025): 105542.

\end{thebibliography}
\end{document}